\documentclass[12pt]{iopart}

\usepackage{iopams}
\usepackage{amsthm}
\usepackage{graphicx,color}

\newcommand{\R}{\mathbb{R}}
\newcommand{\Z}{\mathbb{Z}}
\newcommand{\N}{\mathbb{N}}
\newcommand{\abs}[1]{\left\vert#1\right\vert}
\newcommand{\id}{\mbox{id}}
\newcommand{\im}{\mbox{im}}
\newcommand{\norm}[1]{\Vert#1\Vert}

\newcommand{\bo}{{\boldsymbol{\omega}}}
\newcommand{\bn}{{\boldsymbol{\nu}}}

\newtheorem{thm}{Theorem}[section]
\newtheorem{cor}[thm]{Corollary}
\newtheorem{lem}[thm]{Lemma}
\newtheorem{defn}[thm]{Definition}

\newtheorem{rem}[thm]{Remark}
\newtheorem{hyp}{Hypothesis}[section]

\begin{document}

\title[Heteroclinic chains involving periodic orbits]{Lin's method for
heteroclinic chains involving periodic orbits} 

\author{J\"urgen Knobloch$^1$ and Thorsten Rie\ss$^2$}
\address{$^1$Faculty of
Mathematics and Natural Sciences, Technische Universit\"at Ilmenau, 
PF~100565, 98684 Ilmenau, Germany\\
$^2$Center for Applied Mathematics, 657 Frank H.T. Rhodes Hall,
Cornell University, Ithaca, NY 14853, U.S.A.}
\ead{\mailto{juergen.knobloch@tu-ilmenau.de},
\mailto{triess@cam.cornell.edu}}

\vspace*{5mm}

\begin{abstract}
  We present an extension of the theory known as Lin's method to heteroclinic
  chains that connect hyperbolic equilibria and hyperbolic periodic orbits.
  Based on the construction of a so-called Lin orbit, that is, a sequence of
  continuous partial orbits that only have jumps in a certain prescribed linear
  subspace, estimates for these jumps are derived.  We use the jump estimates to
  discuss bifurcation equations for homoclinic orbits near
  heteroclinic cycles between an equilibrium and a periodic orbit (EtoP cycles).
\end{abstract}


\section{Introduction}
\label{sec:intro}

Connecting cycles involving hyperbolic equilibria and hyperbolic periodic orbits
play an important role in many applications, cf. \cite{Krauskopf2008} and
references therein.  The bifurcation analysis in the vicinity of such a
connecting cycle is crucial for the understanding of the system's behavior.  In
this respect, both the theoretical bifurcation analysis and numerical
implementations are of high interest in current research.

Lin's method has proved to be an appropriate tool for discovering recurrent dynamics
near a given cycle. The method dates from \cite{Lin1990}, where heteroclinic
chains consisting of hyperbolic fixed points, all having the same index
(dimension of the unstable manifold), and heteroclinic orbits connecting them
are considered.  The basic idea of Lin's method is to construct discontinuous
orbits with well defined discontinuities (jumps), so-called Lin orbits, near the
original cycle.  By `making these jumps zero' one finally finds real orbits
staying for all time close to the cycle under consideration.  In 1993 Sandstede
\cite{Sandstede1993} gained jump estimates which allow an effective discussion
of the bifurcation equations.  For a survey of the many applications and
several extensions of the method we refer to \cite{Lin2008}.

In this paper we present an extension of the theory of Lin's method to arbitrary
heteroclinic chains connecting hyperbolic equilibria and hyperbolic periodic
orbits.  Related problems have been studied
in~\cite{Champneys,Rademacher2005,Rademacher2008} and~\cite{Riess2008}; for the numerical
implementation of these ideas we refer to~\cite{Krauskopf2008}.

In the presence of periodic orbits the construction of Lin orbits is much more
involved, because the dynamics near the periodic orbit has to be incorporated.
The handling of the flow near the periodic is the main difference between the
approaches in~\cite{Rademacher2005,Rademacher2008} and~\cite{Riess2008}.
Based on the ideas in~\cite{Riess2008}, we construct certain
partial (discontinuous) orbits running between Poincar\'e sections of two
consecutive periodic orbits of the given chain.  Then the dynamics near each
periodic orbit is described by means of the corresponding Poincar\'e map.
Finally, the different orbits are coupled in the Poincar\'e section in each
case.  This approach allows to apply immediately results from Lin's method for
discrete systems \cite{Knobloch2004}.

We consider a family of ODE
\begin{equation}\label{eq:system}
\dot{x} = f(x,\lambda),f\in C^k(\R^n\times\R^m,\R^n), k\ge 3.
\end{equation}
For a particular parameter value, say $\lambda=0$, we assume that the system has
a heteroclinic chain consisting of hyperbolic periodic orbits $\gamma_i$ and
heteroclinic orbits $q_i$ connecting $\gamma_i$ and $\gamma_{i+1}$.  Here we
explicitely admit that the minimal period of either of these periodic orbits may
be zero, meaning that either of these orbits may be an equilibrium.  We want to
note that, for instance, a heteroclinic cycle between an equilibrium and a
periodic orbit can be considered as such a heteroclinic chain. In this case the
chain consists of copies of the cycle under consideration which are stringed
together.

We refer to a segment $\gamma_i\cup q_i\cup\gamma_{i+1}$ of the given chain as a
{\sl short heteroclinic chain segment}. Near $q_i$ we construct a discontinuous
orbit $X_i$ satisfying certain boundary conditions $(B_i^-)$ near $\gamma_i$ and
$(B_{i+1}^+)$ near $\gamma_{i+1}$.  There the discontinuity is a well defined
jump $\Xi_i$ near $q_i(0)$.  Those orbits we call {\sl short Lin orbit
segments}.

It can be shown that arbitrarily many consecutive short Lin orbit segments can
be linked together to a Lin orbit close to the original chain, see
\cite{Riess2008} for chains related to heteroclinic cycles connecting one
equilibrium and one periodic orbit.  In the present paper we confine
ourselves to linking two consecutive short Lin orbit segments $X_l$ and $X_r$
related to $\gamma_{l}\cup q_{l}\cup\gamma$ and $\gamma\cup q_r\cup\gamma_{r}$
to a {\sl long Lin orbit segment} with boundary conditions $(B_{l}^-)$ and
$(B_{r}^+)$.  Apart from the fact that this procedure reveals the basic idea for
linking arbitrarily many consecutive short Lin orbit segments, it is eligible
for consideration in its own right.  So it suffices to consider long Lin orbit
segments for the detection of $1$-homoclinic orbits near a heteroclinic cycle
connecting two periodic orbits. Here, $1$-homoclinic orbits are characterized by
only one large excursion before returning to their
starting point.

If $\gamma$ is an equilibrium, the existence proof of long Lin orbit segments
runs to large extent parallel to `classical constructions' of Lin's method,
see~\cite{Riess2008}. For that reason we consider only the case that $\gamma$ is a
periodic orbit with nonzero minimal period.  Roughly speaking, the orbits $X_l$
and $X_r$ are linked via an orbit $x$ that defines the behavior of the newly
generated orbit along $\gamma$. We construct $x = x(y)$ as a suspension of a certain
orbit $y$ of an appropriate Poincar\'e map. In this process the boundary
conditions $(B_{l}^-)$ and $(B_{r}^+)$ remain untouched.

Finally we give estimates of the jumps $\Xi_l$ and $\Xi_r$, which allow us to
discuss the bifurcation equations $\Xi_l=0$ and $\Xi_r=0$ for detecting actual
orbit segments near the given long orbit segment $\gamma_l\cup q_l\cup\gamma\cup
q_r\cup\gamma_r$.

We apply our results to study homoclinic orbits near a heteroclinic cycle
connecting a hyperbolic equilibrium $E$ and a hyperbolic periodic orbit $P$
(with nonzero minimal period), an {\sl EtoP cycle} for short. Here we only
consider 1-homoclinic orbits to the equilibrium. Those orbits may differ
considerably in their length of stay near $P$. This length correlates to the
number $\nu$ of rotations the homoclinic orbit performs along $P$ or, in the
above notation, it correlates to the length $\nu$ of the orbit $y$.

Indeed, in numerical computations it has been observed that the homoclinic
orbits for different $\nu$ all lie on the same continuation curve. Moreover,
this continuation curve shows a certain snaking behavior and accumulates on a
curve segment related to the existence of the primary EtoP cycle, cf.
figure~\ref{fig:bif_h1b} panel (a).  The addressed snaking behavior of a system
with reinjection was revealed numerically in \cite{Krauskopf2006} and
\cite{Krauskopf2008}.

In particular, we explain two local phenomena appearing in this global snaking
scenario.  First we consider a codimension-one EtoP cycle. Apart from $E$ and
$P$ this cycle consists of a robust heteroclinic orbit $q_l$ connecting $E$ to
$P$, and a codimension-one heteroclinic orbit $q_r$ connecting $P$ to $E$.
Further, the dimensions of the unstable manifold of $P$ and the stable manifold
of $E$ add up to the space dimension.

Let $\lambda$ be the one-dimensional parameter unfolding $q_r$ and, hence,
unfolding the entire cycle.  In that unfolding we describe the accumulation of
homoclinic orbits at the primary heteroclinic cycle.  More precisely, we prove
that there is a sequence $(\lambda_\nu)_{\nu\in\N}$ tending to zero such that for
each $\lambda_\nu$ there is a homoclinic orbit to the equilibrium, while for
$\lambda=0$ the heteroclinic cycle exists. Moreover, with increasing $\nu$ the
corresponding homoclinic orbits stay longer near the periodic orbit, performing
an increasing number of rotations along the periodic orbit.

In a second scenario we assume that $W^u(E)$ and $W^s(P)$ do no longer intersect
transversally but have a quadratic tangency --- still we assume that $q_r$ is of
codimension one as described above. Then the entire EtoP cycle is of codimension
two.  Let the parameter $\lambda_{l/r}$ unfold the orbits $q_{l/r}$, and assume
that the EtoP cycle exists for $\lambda=(\lambda_l,\lambda_r)=0$.  Then, in the
neighborhood of $\lambda=0$, we find a sequence of curves $\kappa_\nu$ in the
$\lambda$-plane for which a homoclinic orbit to the equilibrium exists. As above
with increasing $\nu$ the corresponding homoclinic orbits stay longer and longer
near the periodic orbit, performing an increasing number of rotations along
the periodic orbit.  For each $\nu$ the curve $\kappa_\nu$ has a turning point
$\lambda_\nu$ tending to zero as $\nu$ tends to infinity.  This explains the
curve progression of $h_1$ in a small neighborhood of $c_1\cap t_0$,
cf.~\fref{fig:bif_h1b}.

The paper is organized as follows. In section~\ref{sec:onelink} we develop Lin's
method for short heteroclinic chain segments. The main theorem in this respect
is theorem~\ref{thm:linbvp}, which states the existence of short Lin orbit
segments. Corollary~\ref{cor:linbvp} extends theorem~\ref{thm:linbvp} to
boundary conditions that enforce that the short Lin orbit segment `starts' in
the unstable manifold of $\gamma_i$. In lemma~\ref{lem:jump_2} we give an
estimate of the jump function, which we extend in corollary~\ref{cor:jump_2} to
the situation of corollary~\ref{cor:linbvp}.
In section~\ref{sec:twolinks} we describe the coupling of two consecutive short
Lin orbit segments $X_l$ and $X_r$ to a long Lin orbit segment. Their existence
is stated in theorem~\ref{thm:lhc}, and corollary~\ref{cor:lhc} extends this
assertion on $X_l$ and $X_r$ asymptotic to $\gamma_l$ and $\gamma_r$,
respectively. The corresponding estimates of the jumps are given in
lemma~\ref{lem:jump} and corollary~\ref{cor:jump}.
In section~\ref{sec:applications} we consider 1-homoclinic orbits near EtoP
cycles. The corollaries~\ref{cor:Xi_r_zeros} and \ref{cor:Xi_r_zeros_n=3}
explain the accumulation of $1$-homoclinic orbits, and
corollary~\ref{cor:app_turning} describes the accumulation of vertices of
continuation curves that are obtained by unfolding a ray (in parameter
space) that is covered twice.

\section{Lin's method for short heteroclinic chains}
\label{sec:onelink}

Consider the ODE \eref{eq:system}.  Throughout this section we assume that for
$\lambda = 0$ there is a short heteroclinic chain segment $\gamma^-\cup q\cup
\gamma^+$ with hyperbolic periodic orbits $\gamma^-$ and $\gamma^+$. We
explicitly admit that the minimal period $T^{-/+}$ of either of them may be
zero, meaning that $\gamma^-$ and/or $\gamma^+$ may be hyperbolic equilibria.
Let $W^{s/u}_\lambda(\gamma^\pm)$ denote the stable/unstable manifolds of
$\gamma^\pm$, and we use the short notation $W^{s/u}(\gamma^\pm)$ for the
corresponding manifolds at $\lambda=0$. Further, $T_qW^{s/u}$ denotes the
tangent space of the corresponding manifold at $q$.

We introduce subspaces $W^+$, $W^-$ and $U$ as follows:
\begin{eqnarray*}
  \left( T_{q(0)} W^{u}(\gamma^-) \cap T_{q(0)} W^{s}(\gamma^+) \right) \oplus W^- =
  T_{q(0)} W^{u}(\gamma^-),\\ \left( T_{q(0)} W^{u}(\gamma^-) \cap T_{q(0)}
  W^{s}(\gamma^+)
  \right) \oplus W^+ = T_{q(0)} W^{s}(\gamma^+)\mbox{ and }\\ \mbox{span}\{
  f(q(0),0) \} \oplus U =  T_{q(0)} W^{u}(\gamma^-) \cap T_{q(0)} W^{s}(\gamma^+).
\end{eqnarray*}
In other words, the linear spaces $W^-$ and $W^+$ are contained in the tangent
spaces of the unstable and stable manifolds of $\gamma^-$ and $\gamma^+$ at
$q(0)$, but do \emph{not} contain their common directions, which are collected
in $\mbox{span}\{ f(q(0),0) \} \oplus U$.  

Using a scalar product $\langle\cdot,\cdot\rangle$ in $\R^n$ we define
\begin{equation}\label{eq:def_Z}
Z:=(W^+ \oplus W^- \oplus U \oplus \mbox{span}\{ f(q(0),0)\})^\perp
\end{equation}
and
\begin{equation}
\label{eq:Y}
  Y = W^+ \oplus W^- \oplus U\oplus Z,
\end{equation}
and we denote the projection onto $U$ in accordance with the direct sum
decomposition \eref{eq:Y} by $P^U$.  Note that either of the involved spaces
$W^\pm$, $U$ and $Z$ may be trivial.  Finally, we define a cross-section
$\Sigma$ of $q$ as
\begin{equation*}
  \Sigma := q(0) + Y.
\end{equation*}

Our goal is to construct `discontinuous orbits' near $q$ that satisfy certain
boundary conditions ($B^-$) and ($B^+$) near $\gamma^-$ and $\gamma^+$. Actually
those orbits consist of two orbit segments where the end point of the first and
the starting point of the second one are in $\Sigma$ and their difference is in
$Z$, which is reflected in the boundary condition ($J$).

In a first step we prove the existence of orbit segments that lie in the unstable
and stable manifolds of $\gamma^-$ and $\gamma^+$, respectively, and that
satisfy certain jump conditions in~$\Sigma$.
\begin{thm}
  \label{thm:splitting}
  There is a constant $c>0$ such that for $\abs{\lambda}<c$ and $u\in U$,
  $\abs{u}<c$,
  there is a unique pair of solutions
  $\left(q^-(u,\lambda),q^+(u,\lambda)\right)$ of~\eref{eq:system} that satisfy
  \begin{enumerate}
    \item $q^+(u,\lambda)(0)\in W_\lambda^s(\gamma^+)$, $q^-(u,\lambda)(0)\in
      W_\lambda^u(\gamma^-)$, 
    \item $q^+(u,\lambda)(0),q^-(u,\lambda)(0)\in\Sigma$,
    \item $P^U\left(q^+(u,\lambda)(0) - q(0)\right) =
      P^U\left(q^-(u,\lambda)(0)-q(0)\right) = u$ and
    \item $q^+(u,\lambda)(0) - q^-(u,\lambda)(0)\in Z$.
  \end{enumerate}
\end{thm}

Figures~\ref{fig:splitting} and~\ref{fig:uinsigma} give a graphical interpretation of
theorem~\ref{thm:splitting}, the proof is given in \sref{sec:splitting}.

\begin{figure}[ht]
\begin{center}
\includegraphics[scale=.8]{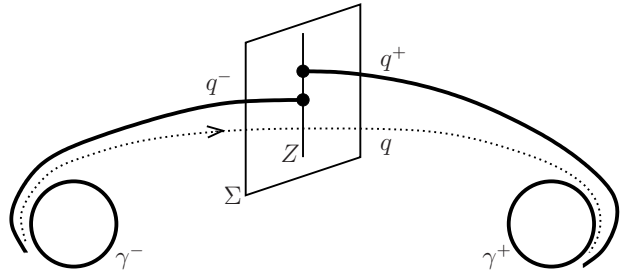} 
\caption{\label{fig:splitting}Sketch of the
situation described in theorem~\ref{thm:splitting}, showing the orbits $q^-
\subset W_\lambda^u(\gamma^-)$ and $q^+ \subset W_\lambda^s(\gamma^+)$.  Within the
cross-section $\Sigma$, the two orbits have a jump in the direction $Z$.  Note that
$\gamma^-$ and $\gamma^+$ are depicted as periodic orbits, but either of them may
be an equilibrium.  The dotted connection $q$ is present for $\lambda = 0$.}
\end{center}
\end{figure}

\begin{figure}[ht]
\begin{center}
\includegraphics[scale=.8]{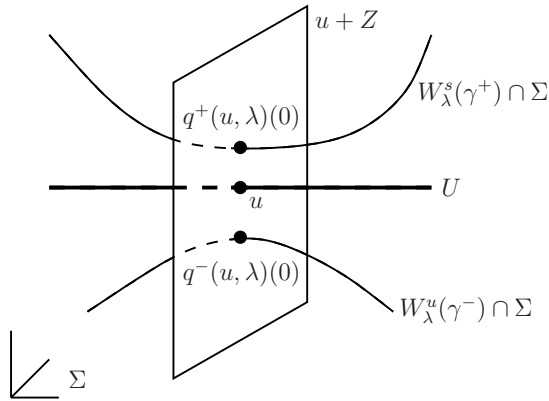}
\caption{\label{fig:uinsigma}The situation
inside $\Sigma$, depicted are the traces of $q^{-/+}(u,\lambda)$,
$W_\lambda^{u}(\gamma^-)$, $W_\lambda^{s}(\gamma^+)$
and the direction $U$ consisting of the common tangent directions (restricted to
$\Sigma$).  The depicted situation corresponds to a `quadratic tangency' of
$W^{u}(\gamma^-)$ and $W^{s}(\gamma^+)$.}
\end{center}
\end{figure}

In the next step, we perturb the solutions $q^\pm$ given by
theorem~\ref{thm:splitting} to construct solutions that stay near the connecting
orbit $q$ and satisfy projection boundary conditions near $\gamma^-$ and
$\gamma^+$.  Moreover, these solutions are also allowed to have a jump in
the direction $Z$ (within $\Sigma$).

To formulate the boundary conditions we define projections
$P^\pm(u,\lambda)(t)$ by
\begin{equation}
  \label{eq:projectionplus}
  \eqalign{
  \im P^+(u,\lambda)(0) &= T_{q^+(u,\lambda)(0)} W_\lambda^s(\gamma^+),\\
  \ker P^+(u,\lambda)(0) &= W^- \oplus Z\mbox{ and }\\
  P^+(u,\lambda)(t) &:= \Phi^+(t,0) (P^+(u,\lambda)(0))\Phi^+(0,t),\, t\in\R^+,
  }
\end{equation}
and analogously
\begin{equation}
  \label{eq:projectionminus}
  \eqalign{
  \im P^-(u,\lambda)(0) &= T_{q^-(u,\lambda)(0)} W_\lambda^u(\gamma^-)\\
  \ker P^-(u,\lambda)(0) &= W^+ \oplus Z\mbox{ and }\\
  P^-(u,\lambda)(t) &:= \Phi^-(t,0) (P^-(u,\lambda)(0))\Phi^-(0,t),\, t\in\R^-.
  }
\end{equation}
Here $\Phi^\pm(\cdot,\cdot)=\Phi^\pm(u,\lambda)(\cdot,\cdot)$ denotes the transition matrix of the variational
equation along $q^\pm(u,\lambda)(\cdot)$.

Throughout we denote by $\abs{\cdot}$ the absolute value of a number or the the
Euclidian norm of an $n$-tuple.  For elements $a =\left( a^-,a^+ \right)\in \R^n\times\R^n$ we define
$\|a\|:=\max\{|a^-|,|a^+|\}$.
\begin{thm}
\label{thm:linbvp}
Fix $\omega^-,\omega^+>0$.
There is a constant $c>0$ such
that for $\abs{\lambda}<c$, $u\in U$, $\abs{u}<c$, and $a =
\left( a^-,a^+ \right)\in \R^n\times\R^n$, $\norm{a}<c$, 
there is a unique pair of solutions
$\left( x^-,x^+ \right)$ of~\eref{eq:system} that satisfies
\begin{description}
  \item[($J$)] $x^{-/+}(a,u,\lambda)(0)\in\Sigma$,
  \quad 
$x^-(a,u,\lambda)(0) -
    x^+(a,u,\lambda)(0) \in Z$,
  \item[($B^-$)] $\left( \id - P^-(u,\lambda)(-\omega^-) \right) \left(
    x^-(a,u,\lambda)(-\omega^-) -
    q^-(u,\lambda)(-\omega^-)- a^-\right) = 0$,
  \item[($B^+$)] $\left( \id -
    P^+(u,\lambda)(\omega^+) \right) \left(
    x^+(a,u,\lambda)(\omega^+) - q^+(u,\lambda)(\omega^+)-a^+
    \right) = 0$.
\end{description}
Moreover, $\left( x^-,x^+ \right)=\left( x^-,x^+ \right)(a,u,\lambda)$ depends
smoothly on $(a,u,\lambda)$.
\end{thm}

\begin{figure}[h]
\begin{center}
\includegraphics[scale=.8]{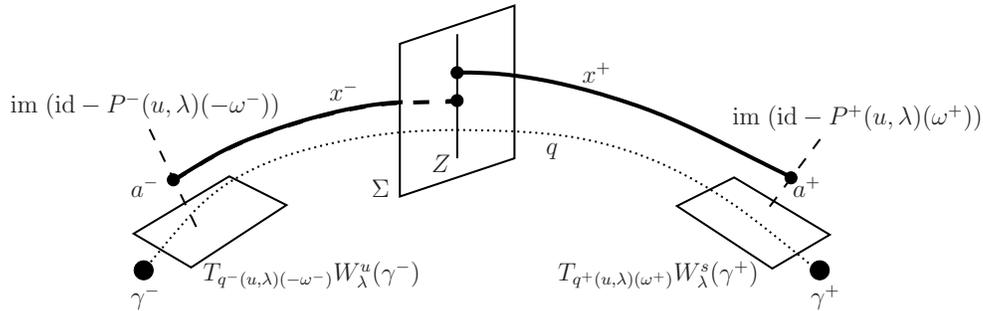} 
\caption{\label{fig:shc}Sketch of a short
Lin orbit segment $(x^-,x^+)$ near a short heteroclinic chain segment $\gamma^-\cup
q\cup\gamma^+$.}
\end{center}
\end{figure}
In other words, theorem~\ref{thm:linbvp} states the existence  of short Lin
orbits segments $X=( x^-,x^+)$ with boundary conditions ($B^-$) and ($B^+$) for
sufficiently small $(a,u,\lambda)$, for a sketch of this situation see
figure~\ref{fig:shc}.  Note that $X$ also depends on $\omega^-,\omega^+>0$,
which we suppress from our notation in this section.  The proof is given in
section~\ref{sec:exptrich}.

\subsection{Orbits in the stable/unstable manifolds --- the proof of
theorem~\ref{thm:splitting}}
\label{sec:splitting}

From a geometrical point of view, the statement of theorem~\ref{thm:splitting}
is rather clear.  Hence we give a proof which exploits the geometry of the traces
of the involved manifolds in~$\Sigma$, see \fref{fig:uinsigma}.
By $B_X(x,r)$ we denote a closed ball in $X$ centered at $x$ with radius $r$. If
the space $X$ is clear from the context we will also write $B(x,r)$ for short.

\begin{proof}[Proof of theorem~\ref{thm:splitting}]
  Using the direct sum decomposition~\eref{eq:Y} we find the following
  representations of the traces in $\Sigma$ of the stable/unstable manifolds of
  $\gamma^{-/+}$ locally around $q(0)$:  For $\varepsilon$ sufficiently small
  and for $w^{s}\in W_\lambda^{s}(\gamma^{+})\cap \Sigma \cap B(q(0),\varepsilon)$ there
  are smooth functions $\tilde{w}^-:W^+\times U\times \R^m \to W^-$ with
  $\tilde{w}^-(0,0,0)=0$, $D_1\tilde{w}^-(0,0,0)=0$ and $z^+:W^+\times U\times
  \R^m \to Z$ such that
  \begin{equation*}
    \label{eq:sp_v+}    
      w^s= q(0)+ w^+ + \tilde{w}^-(w^+,u^+,\lambda)
      + z^+(w^+,u^+,\lambda) + u^+.
  \end{equation*}
  Similarly for $w^{u}\in W_\lambda^{u}(\gamma^{-})\cap \Sigma \cap
  B(q(0),\varepsilon)$ there are smooth functions $\tilde{w}^+:W^-\times
  U\times \R^m \to W^+$ with $\tilde{w}^+(0,0,0)=0$, $D_1\tilde{w}^+(0,0,0)=0$
  and $z^-:W^-\times U\times \R^m \to Z$ such that
  \begin{equation*}
    \label{eq:sp_v-}
      w^u =  q(0)+ \tilde{w}^+(w^-,u^-,\lambda) + w^- +
      z^-(w^-,u^-,\lambda) + u^-.
  \end{equation*}
  The demand that $q^+(u,\lambda)(0) - q^-(u,\lambda)(0) \in Z$ results in $u^- =
  u^+ =: u$ and
  \begin{equation*}
    \eqalign{
      w^+ &= \tilde{w}^+(w^-,u,\lambda),\\
      w^- &= \tilde{w}^-(w^+,u,\lambda),
    }
  \end{equation*}
  which then can be solved for $(w^+,w^-) =
  (\hat{w}^+(u,\lambda),\hat{w}^-(u,\lambda))$ around $(u,\lambda) = (0,0)$.

  Now we get $q^{+}(u,\lambda)$ as the solution of the initial the value problem 
  \begin{equation*}
    \eqalign{ \dot{x}&=f(x,\lambda)\\ x(0)&= q(0)+ \hat{w}^{+}(u,\lambda) +
    \tilde{w}^-(\hat{w}^{+}(u,\lambda),u,\lambda) +
    z^+(\hat{w}^{+}(u,\lambda),u,\lambda) + u, }
  \end{equation*}
  and  $q^{+}(u,\lambda)$ is the solution of the initial the value problem 
  \begin{equation*}
    \eqalign{
    \dot{x}&=f(x,\lambda)\\
    x(0)&= q(0)+ \tilde{w}^+(\hat{w}^-(u,\lambda),u,\lambda) + \hat{w}^-(u,\lambda) +
      z^-(\hat{w}^-(u,\lambda),u,\lambda) + u.
    }
  \end{equation*}
\end{proof}

\subsection{Short Lin orbit segments --- the proof of theorem~\ref{thm:linbvp}}
\label{sec:exptrich}

In this section we give a detailed proof of theorem~\ref{thm:linbvp}.  This
proof is based on the ideas of Lin's method, but used in a slightly different
way.  The main difference to the `classical' proof of Lin's method is that
we keep the boundary conditions near $\gamma^-$ and $\gamma^+$ as linear
projection conditions, while finding solutions of the full nonlinear system that
additionally satisfy certain jump conditions.

We start with a lemma that provides some properties of the projections
$P^\pm(u,\lambda)(\cdot)$, as introduced in~\eref{eq:projectionplus}
and~\eref{eq:projectionminus}, that are used in the proofs throughout
this section.

\begin{lem}
  \label{lem:exptrich}
  There are projections $P_s^+(u,\lambda)(\cdot)$ and $P_c^+(u,\lambda)(\cdot)$
  such that
  \begin{equation*}
    \eqalign{
    P^+(u,\lambda)(t) &= P_s^+(u,\lambda)(t) + P_c^+(u,\lambda)(t) \mbox{ for
    } t\in\R^+.
    }
  \end{equation*}
  The projections $P_{s/c}^+(u,\lambda)(\cdot)$ satisfy the following:
  \begin{enumerate}
   \item $\Phi(t,\tau) P_{s/c}^+(\tau) = P_{s/c}^+(t) \Phi(t,\tau) \quad \forall
        t,\tau \in \R^+$,
    \item there are constants $K>0$, $\delta^s,\delta^u > \delta^c \ge 0$ such
      that
      \begin{eqnarray*}
        \abs{ \Phi(t,\tau) P^+_s(\tau) } \le K e^{-\delta^s (t-\tau)} & ,t \ge
        \tau,\\
        \abs{ \Phi(t,\tau) P^+_c(\tau) } \le K e^{ \delta^c (t-\tau)} & ,t \ge
        \tau,\\
        \abs{ \Phi(t,\tau) P^+_c(\tau) } \le K e^{-\delta^c (t-\tau)} & ,\tau \ge
        t,\\
        \abs{ \Phi(t,\tau) (\id - P^+(\tau) ) } \le K e^{-\delta^u (\tau-t)} & ,\tau \ge t.
      \end{eqnarray*}

  \end{enumerate}
\end{lem}
This lemma follows immediately from the fact that the variational equation along
the solutions $q^+(u,\lambda)$ has an exponential trichotomy on
$\R^+$~\cite{Beyn1994,HL86,CL90}.  Note that the exponents $\delta^{s/c/u}$
are determined by the eigenvalues/Floquet exponents of $\gamma^+$.  Since
$\gamma^+$ is a hyperbolic periodic orbit, we have $\delta^c = 0$~\cite{Beyn1994}.  We
want to note explicitely that the images of $P^+_s(u,\lambda)(t)$ are
well-determined --- these are the tangent spaces of the strong stable fiber of
$\gamma^+$ at $q^+(u,\lambda)(t)$.  Also note that if $\gamma^+$ is a
hyperbolic equilibrium, the variational equation along $q^+(u,\lambda)$ has in
fact an exponential dichotomy \cite{Coppel1978}, i.e. $P_c^+(u,\lambda) = 0$,
and $\im P^+_s(u,\lambda)(t)=T_{q^+(u,\lambda)(t)}W_\lambda^s(\gamma^+)$.

For the projection $P^-(u,\lambda)$ an analogous lemma holds, exploiting
that the variational equation along the solutions $q^-(u,\lambda)$ has an
exponential trichotomy on $\R^-$: \[ P^-(u,\lambda)(t) = P_u^-(u,\lambda)(t) +
P_c^-(u,\lambda)(t) \mbox{ for } t\in\R^-, \] where $P_u^-(u,\lambda)(t)$
projects on the tangent space of the strong unstable fiber at
$q^-(u,\lambda)(t)$.

To prove theorem~\ref{thm:linbvp} we consider small perturbations of the
solutions $q^{-/+}(u,\lambda)$:
\begin{eqnarray}\label{eq:setting_x}
\eqalign{
  x^-(t) &:= q^-(u,\lambda)(t) + v^-(t),\quad t\in\R^-,\\
  x^+(t) &:= q^+(u,\lambda)(t) + v^+(t),\quad t\in\R^+.
}
\end{eqnarray}
The perturbation terms $v^-$ and $v^+$ are solutions of 
\begin{equation}
\label{eq:nonlinvar}
\eqalign{
  \dot{v}^- &= D_1 f(q^{-}(u,\lambda)(t),\lambda)v^- + h^-(t,v^-,u,\lambda),\\
  \dot{v}^+ &= D_1 f(q^{+}(u,\lambda)(t),\lambda)v^+ + h^+(t,v^+,u,\lambda),
}
\end{equation}
where $h^{\pm}(t,v,u,\lambda) := f(q^{\pm}(u,\lambda)(t)+v,\lambda) -
f(q^{\pm}(u,\lambda)(t),\lambda) - D_1 f(q^{\pm}(u,\lambda)(t),\lambda)$.  The
boundary conditions ($J$) and ($B^\pm$) for $x^\pm$ yield boundary
conditions for $v^\pm$:
\begin{description}
  \item[($J_v$)] ${v}^\pm(0) \in W^-\oplus W^+ \oplus Z$,\quad
    ${v}^-(0) - {v}^+(0) \in Z$, 
  \item[($B_v^-$)]
    $\left(\id-P^-(u,\lambda)(-\omega^-)\right) {v}^-(-\omega^-)
    = \left(\id-P^-(u,\lambda)(-\omega^-)\right)a^-$, 
  \item[($B_v^+$)]
    $\left(\id-P^+(u,\lambda)(\omega^+)\right) {v}^+(\omega^+) =
    \left(\id-P^+(u,\lambda)(\omega^+)\right)a^+$.
\end{description}

In a first approximation of~\eref{eq:nonlinvar}, we replace the functions
$h^-$ and $h^+$ by functions $g^-$ and $g^+$ only depending on $t$:
\begin{equation}
\label{eq:nonlinvarg}
\eqalign{
  \dot{v}^- &= D_1 f(q^{-}(u,\lambda)(t),\lambda)v^- + g^-(t),\cr
  \dot{v}^+ &= D_1 f(q^{+}(u,\lambda)(t),\lambda)v^+ + g^+(t).
}
\end{equation}

For given $\omega^-$ and $\omega^+$ we write $\bo:=(\omega^-,\omega^+)$, and 
we introduce the space $V_\bo$ of pairs of continuous functions as
\begin{equation*}
  V_\bo := \left\lbrace \left(v^-,v^+\right) : v^- \in
  C\left(\left[-\omega^-,0\right],\R^n\right) \mbox{ and } v^+ \in
  C\left(\left[0,\omega^+\right],\R^n\right) \right\rbrace.
\end{equation*}
We equip $V_\bo$ with the norm
$\norm{\left(v^-,v^+\right)} := \max\left\lbrace \norm{v^-},\norm{v^+}
\right\rbrace$, where $\norm{v^\pm}$ denotes the supremum norm.

We actually perform the proof of theorem~\ref{thm:linbvp} in two steps. First we
consider the `linearized' equation \eref{eq:nonlinvarg} (linearized in the sense
that $g^\pm$ does not depend on $v^\pm$) and show that there are unique
solutions $\hat v^\pm$ satisfying boundary conditions ($J_v$), ($B_v^-$) and
($B_v^+$); see Lemma~\ref{lem:link_linear} below.  Of course $\hat v^\pm$
depend (among others) on $g^\pm$.  In the next step we replace the function
$g^\pm$ in these dependencies by $h^\pm$. This gives a fixed point equation (see
\eref{eq:fixedpoint} below) that is equivalent to \eref{eq:nonlinvar} with
boundary conditions ($J_v$), ($B_v^-$) and ($B_v^+$).

\begin{lem}
\label{lem:link_linear}
Let $u$ and $\lambda$ be in accordance with theorem~\ref{thm:splitting}, and let
$\bo$ be fixed.  Then, for given $a=\left(a^-,a^+\right)\in \R^n\times\R^n$ and
$g=\left(g^-,g^+\right)\in V_\bo$ there is a unique pair of solutions $\hat{v} =
\left( \hat{v}^-,\hat{v}^+ \right)\in V_\bo$, $\hat{v} =\hat{v}(g,a,u,\lambda)$,
of~\eref{eq:nonlinvarg} that satisfy boundary conditions ($J_v$), ($B_v^-$) and
($B_v^+$).  \\ For fixed $u$ and $\lambda$ the pair of solutions $\hat{v}$ depends linearly on
$(g,a)$, and it depends smoothly on $(g,a,u,\lambda)$.  Moreover, there are
constants $\hat{C}_a,\,\hat{C}_g >0$ such that
  \begin{eqnarray}
    \label{eq:normv}
    \norm{\hat{v}(g,a,u,\lambda)} \le \hat{C}_a 
    \norm{a} + \hat{C}_g \norm{g}.
  \end{eqnarray}
The constant $\hat{C}_a$ is uniform in $\omega^+$ and $\omega^{-}$, while
$\hat{C}_g$ is uniform in $\omega^\pm$ only if $\gamma^\pm$ is an equilibrium.
\end{lem}

\begin{proof}
  For a shorter notation in this proof we omit the dependencies of $\Phi$ and
  $P^\pm$ on $u$ and $\lambda$.  The variation of constants formula gives
  \begin{equation}
    \label{eq:varconst}
    \eqalign{
      v^-(t) &= \Phi^-(t,0) v^-(0) + \int_0^t \Phi^-(t,\tau)g^-(\tau) d\tau,\\
      v^+(t) &= \Phi^+(t,0) v^+(0) + \int_0^t \Phi^+(t,\tau)g^+(\tau) d\tau
    }
  \end{equation}
  as solutions of \eref{eq:nonlinvarg}, which can be transformed to
  \begin{eqnarray*}
    \left( \id-P^-(0) \right) v^-(0) &=& \Phi^-(0,-\omega^-) \left(
    \id-P^-(-\omega^-) \right) v^-(-\omega^-)\\
    &&\;\; + \int_{-\omega^-}^0
    \Phi^-(0,\tau)\left( \id - P^-(\tau) \right) g^-(\tau) d\tau,\\
    \left(
    \id-P^+(0) \right) v^+(0) &=& \Phi^+(0,\omega^+) \left( \id-P^+(\omega^+)
    \right) v^+(\omega^+)\\
    &&\;\; - \int_0^{\omega^+} \Phi^+(0,\tau)\left( \id -
    P^+(\tau) \right) g^+(\tau) d\tau.
  \end{eqnarray*}
  Thus
  \begin{equation}
  \label{eq:linop}
  \eqalign{
    \left( \id-P^-(0) \right) v^-(0) &= \Phi^-(0,-\omega^-) \left(
    \id-P^-(-\omega^-) \right)  a^- \\
    &\;\;\;\; +\int_{-\omega^-}^0 \Phi^-(0,\tau)\left( \id - P^-(\tau) \right) g^-(\tau)
    d\tau,\\
    \left( \id-P^+(0) \right) v^+(0) &= \Phi^+(0,\omega^+) \left( \id-P^+(\omega^+)
    \right) a^+ \\
   &\;\;\;\;  -\int_0^{\omega^+} \Phi^+(0,\tau)\left( \id - P^+(\tau) \right) g^+(\tau)
    d\tau.
  }
  \end{equation}

  We decompose $v^-(0),v^+(0)$ in accordance with the boundary condition ($J_\nu$)
  \begin{eqnarray*}
    v^-(0) &=& w^- + w^+ + z^-,\\
    v^+(0) &=& w^- + w^+ + z^+,
  \end{eqnarray*}
  where $w^-\in W^-$, $w^+\in W^+$ and $z^\pm \in Z$.

  With that the left-hand side of~\eref{eq:linop} can be considered as a linear mapping
  \begin{equation*}
\begin{array}{lccc}
    L :& W^+ \times W^- \times Z \times Z &\to& (W^+ \oplus Z)\times(W^- \oplus Z)
\\[1ex]
&(w^+,w^-,z^+,z^-)&\mapsto& \left(\begin{array}{c}
      ( w^+ + z^- )\\
      ( w^- + z^+ )
    \end{array}\right),
\end{array}
  \end{equation*}
  which is invertible. Hence we can solve~\eref{eq:linop} for $v^+(0),v^-(0)$ linearly
  depending on $(g,a)$.  Incorporating the dependence on $(u,\lambda)$ finally gives
  a solution $\hat{v} = \hat{v}(g,a,u,\lambda)$ of \eref{eq:nonlinvarg}
  satisfying the boundary conditions ($J_\nu$), ($B_\nu^-$) and ($B_\nu^+$).
  Note that $\hat{v}$ depends linearly on $(g,a)$, and smoothly on
  $(g,a,u,\lambda)$.

  To prove estimate~\eref{eq:normv}, we decompose $\hat{v}^+$ by means of the
  projection $P^+$:
  \begin{equation*}
    \hat{v}^+(t) = (\id - P^+(t))\hat{v}^+(t) + P^+(t)\hat{v}^+(t).
  \end{equation*}
  Thus we have
  \begin{equation}
    \label{eq:ou_twoterms}
    \abs{\hat{v}^+(t)} \le \abs{ (\id - P^+(t))\hat{v}^+(t) } + \abs{
    P^+(t)\hat{v}^+(t) }.
  \end{equation}

  We use the variation of constants formula and the estimates of
  lemma~\ref{lem:exptrich}
  to derive an estimate for the second term of~\eref{eq:ou_twoterms}:
  \begin{equation*}
    \eqalign{
      \abs{ P^+(t)\hat{v}^+(t) } &= \abs{ P^+(t)\left( \Phi^+(t,0)\hat{v}^+(0)
      + \int_0^t \Phi^+(t,\tau)g^+(\tau) d\tau \right) }\\ 
      &\le K (e^{-\delta^s t} + e^{\delta^c
      t}) \abs{ \hat{v}^+(0) } + M \norm{g^+}\\ 
      &\le K (e^{-\delta^s t} + 1)
      \abs{ \hat{v}^+(0) } + M \norm{g^+}.
    }
  \end{equation*}
  The constants $\delta^s$, $\delta^c$ and $K$ are the corresponding constants
  of the exponential trichotomy ($\delta^s>\delta^c =0$;
  see lemma~\ref{lem:exptrich} and the remarks above).
  Note that if $\gamma^+$ is not an equilibrium, the
  constant $M$ depends on $\omega^+$, in fact $M\to\infty$ as
  $\omega^+\to\infty$ in the same order as $\omega^+$.

  We estimate $\abs{ \hat{v}^+(0) }$ by applying $L^{-1}$ to~\eref{eq:linop}
  and using lemma~\ref{lem:exptrich} once again:
  \begin{equation}
    \label{eq:ou_normv0}
    \abs{ \hat{v}^+(0) } \le \norm{L^{-1}} \tilde{K} \left(\abs{a^+} +
    \abs{a^-}\right) + \hat M \norm{ (g^+, g^-) }.
  \end{equation}
  Here, the constant $\hat M$ is uniform in $\omega^+$ and
  $\omega^-$.

Thus we have
  \begin{equation*}
    \abs{ P^+(t)v^+(t) } \le C_{1,a} \norm{a} + C_{1,g}\norm{g}).
  \end{equation*}
The constant $C_{1,a}$ is uniform in $\omega^+$ and $\omega^-$, while 
$C_{1,g}(\bo)$ tends to infinity in the same order as $\|\bo\|$.

  For the first term of the right hand side of~\eref{eq:ou_twoterms} we use
  \begin{equation*}
    \eqalign{
      (\id - P^+(t))\hat{v}^+(t) =& \Phi^+(t,\omega^+)(\id -
      P^+(\omega^+))a^+\\
      &\quad - \int_t^{\omega^+} \Phi^+(t,\tau) (\id - P^+(\tau)) g^+(\tau)
      d\tau
    }
  \end{equation*}
  and thus finally get
  \begin{equation*}
    \abs{ (\id - P^+(t)) \hat{v}^+(t) } \le C_2 (\abs{a^+} + \norm{g}).
  \end{equation*}
Note that $C_2$ is uniform in $\omega^+$ (and does not depend on $\omega^-$).

Summarizing, there are constants $\hat C^+_a$ and $\hat C^+_g$
such that
\[
\norm{\hat v^+(g,a,u,\lambda)}\le \hat C^+_a\norm{a}+ \hat C^+_g\norm{g},
\]
where $\hat C^+_a$ is uniform in $\omega^\pm$ and $\hat C^+_g(\bo)$ tends to infinity in the same order as $\|\bo\|$.
  Proceeding  with $\hat{v}^-$ in a similar way yields
  estimate~\eref{eq:normv}.
\end{proof}

\begin{lem}\label{lem:link_linear_b}
Let the assumption of lemma~\ref{lem:link_linear} hold.  We define functions
  \begin{equation*}
    \eqalign{
      \hat{a}_\perp^+(g,a,u,\lambda) &:=
      P_s^+(u,\lambda)(\omega^+)\hat{v}^+(g,a,u,
      \lambda)(\omega^+),\\
      \hat{a}_\perp^-(g,a,u,\lambda) &:=
      P_u^-(u,\lambda)(-\omega^-)\hat{v}^-(g,a,u,
      \lambda)(-\omega^-).
    }
  \end{equation*}
  There are constants $\delta^s,\delta^u>0$ and $\hat C>0$ such that 
  \begin{equation}
    \label{eq:est_aperp1}
    \abs{\hat{a}_\perp^+} \le \hat{C}\left( e^{-\delta^s\omega^+}  \norm{a} 
       + \norm{g}\right),\quad
      \abs{\hat{a}_\perp^-} \le \hat{C}\left( e^{-\delta^u\omega^-} \norm{a} 
      + \norm{g}\right).
  \end{equation}
  For the derivatives of $\hat{a}_\perp^+$ and $\hat{a}_\perp^-$ the following 
  estimates hold:

  \begin{equation}
    \label{eq:est_aperpder1}
    \norm{D_2 \hat{a}_\perp^+(g,a,u,\lambda)} \le \hat{C}
      e^{-\delta^s\omega^+},\quad 
      \norm{D_2 \hat{a}_\perp^-(g,a,u,\lambda)} \le \hat{C} e^{-\delta^u\omega^-}
  \end{equation}
  and
 \begin{equation}
    \label{eq:est_aperpder2}
    \norm{D_1 \hat{a}_\perp^+(g,a,u,\lambda)}, \norm{D_1 \hat{a}_\perp^-(g,a,u,\lambda)} \le \hat{C}.
\end{equation}

\end{lem}

\begin{proof}
  For estimate~\eref{eq:est_aperp1} we use lemma~\ref{lem:exptrich} again:
  \begin{equation*}
    \fl\eqalign{
    \abs{\hat{a}_\perp^+(g,a,u,\lambda)} &= \abs{P_s^+(\omega^+)\Big(
      \Phi^+(\omega^+,0)\hat{v}^+(0)+\int_0^{\omega^+}
      \Phi^+(\omega^+,\tau)g^+(\tau) d\tau \Big)}\\ 
      &\le
      Ke^{-\delta^s\omega^+}\abs{\hat{v}^+(0)} + \hat M\norm{g^+}\\ 
      &\le
      C\big( e^{-\delta^s\omega^+}\norm{a} +
      \norm{g}\big).
    }
  \end{equation*}

  For the derivative we note that the dependencies of $\hat{v}^\pm$ on $(g,a)$
  are linear.  Hence, there are linear operators $\tilde{L}^\pm$ depending on
  $u$ and $\lambda$ such that $\hat{v}^\pm(g,a,u,\lambda) =
  \tilde{L}^\pm(u,\lambda)(g,a) = \tilde{L}^\pm(u,\lambda)(g,0) +
  \tilde{L}^\pm(u,\lambda)(0,a)$.  Due to their definition,
  $\hat{a}_\perp^\pm$ also depend linearly on $(g,a)$.  Thus the estimates
  \eref{eq:est_aperpder1} and \eref{eq:est_aperpder2} follow.
\end{proof}

Now we look for solutions of~\eref{eq:nonlinvar} satisfying the boundary
conditions ($J_v$), ($B_v^-$) and ($B_v^+$). For that purpose, in $\hat
v(g,a,u,\lambda)$ we formally replace the function $g$ by $h$.  To perform this
substitution we define the Nemytskii operator $H := \left( H^-,H^+ \right)$:
\begin{equation}\label{eq:nem_op_H}
  \begin{array}{ccl}
  H : V_\bo \times U \times \R^m &\to& \qquad V_\bo
\\
  \qquad(v,u,\lambda) &\mapsto& \left( 
  H^-\left(v^-,u,\lambda\right),H^+\left(v^+,u,\lambda\right) \right),
\end{array}
\end{equation}
where
\[
 H^-(v, u, \lambda)(\cdot) :=h^-\left(\cdot,v(\cdot),u,\lambda\right), \quad
H^+(v, u, \lambda)(\cdot) :=h^+\left(\cdot,v(\cdot),u,\lambda\right).
\]
In \cite{Sandstede1993} it is verified that $H$ has the stated mapping
properties and that $H$ is smooth.

Summarizing, we find that a function $v$ solves the boundary value problem
(\eref{eq:nonlinvar}, ($J_v$), ($B_v^-$), ($B_v^+$)) if and only if it satisfies
the following fixed point equation in $V_\bo$:
\begin{equation}
  \label{eq:fixedpoint}
    v =
    \hat{v}(H(v,u,\lambda),a,u,
    \lambda) =:F(v,a,u,\lambda).
\end{equation}
Note that
\[
 F:V_\bo\times(\R^n\times\R^n)\times U\times\R^m\to V_\bo.
\]

\begin{lem}
  \label{lem:link_v}
Fix some $\bo$.
  There are functions
  $\bar{\varepsilon},\tilde{c},\bar{c},\Omega:\R\to\R^+$ 
such that for all $K>1$ the fixed point problem~\eref{eq:fixedpoint} has a
  unique solution $v = ( v^-,v^+) \in V_\bo$, $v=v(a,u,\lambda)$, in an
  $\bar{\varepsilon}(K)$-neighborhood of $0\in V_\bo$, provided that
$\abs{\lambda},\abs{u}<\bar{c}(K)$, 
and
  $\norm{a}<\tilde{c}(K)$. The solution $v$ depends smoothly on $(a,u,\lambda)$.
\end{lem}

\begin{proof}
  We use the Banach fixed point theorem to prove existence and uniqueness of
  $v$.  First we show
  that there is an $F$-invariant closed ball $B(0,\bar{\varepsilon})\subset
  V_\bo$ and then that $F$ is a
  contraction on $B(0,\bar{\varepsilon})$ with respect to $v$.

Let $\hat{C}_a$, $\hat{C}_g$ and $\hat{C}$ be the constants in accordance with
lemma~\ref{lem:link_linear} and lemma~\ref{lem:link_linear_b}, respectively.
Define $C:=\max\{1,\hat{C},\hat{C}_a,\hat{C}_g\}$. Then, according to
\eref{eq:normv},
  \begin{equation}
\label{eq:normF}
    \norm{F(v,a,u,\lambda)} \le C \left( \norm{a} + \norm{H(v,u,\lambda)} \right).
  \end{equation}
  We start with an estimate for $\norm{H(v,u,\lambda)}$.  From the definition of
  $h^\pm$ we
  see that $H(0,0,0)=0$ and thus we can use the mean value theorem to get the estimate
  \begin{eqnarray}\label{eq:normH}
    \norm{H(v,u,\lambda)} 
    &\le \int_0^1 \norm{ D_1 H(s(v,u,\lambda)) } ds
    \norm{v} \nonumber\\
    &\qquad\qquad + \int_0^1 \norm{ D_2 H(s(v,u,\lambda))
    } ds \abs{u} \nonumber\\
    &\qquad\qquad + \int_0^1 \norm{ D_3 H(s(v,u,\lambda))
    } ds \abs{\lambda}.
  \end{eqnarray}
Applying the mean value theorem to $D_1H$ we find that there is an appropriate constant $D>0$ such that with
\begin{equation}\label{eq:def_bar_epsilon}
 \bar{\varepsilon}(K) := \frac{1}{K^2 C^2 D}
\end{equation}
the following holds: If $\abs{\lambda},\norm{v},\abs{u}<\bar\varepsilon(K)$ then, since $D_1 H(0,0,0) = 0$,
\begin{equation}\label{eq:ou_H}
    \norm{D_1 H(v,u,\lambda)} \le \frac{1}{7K^2C^2}.
  \end{equation}
Further, there is a constant $E$ such that for all $v$, $u$ and $\lambda$ taken from some neighborhood of the origin
  \begin{equation*}\label{eq:est_D23_H}
    \int_0^1 \norm{D_2 H(s(v,u,\lambda))} ds < E, \qquad
    \int_0^1 \norm{D_3 H(s(v,u,\lambda))} ds < E.
  \end{equation*}
By means of $\bar{\varepsilon}$ we further define
\begin{equation}\label{eq:def_tilde_c}
\bar{c}(K):=\frac{\bar{\varepsilon}(K)}{2\cdot 7K^2C^2 E},\qquad
\tilde{c} := \frac{5\bar{\varepsilon}(K)}{7KC}.
\end{equation}
Therefore, we find for $\norm{v}<\bar{\varepsilon}$, $\abs{\lambda},\abs{u}<\bar{c}$ and $K>1$ (recall $C\ge 1$)
 \begin{equation*}
    \norm{H(v,u,\lambda)}\le \frac{2\bar{\varepsilon}}{7K^2C^2} \le
    \frac{2\bar{\varepsilon}}{7KC}.
  \end{equation*}

Finally, estimate \eref{eq:normF} yields that for all
$\abs{\lambda},\abs{u}<\bar{c}$ and $\norm{a}<\tilde{c}$ the mapping
$F(\cdot,a,u,\lambda)$ leaves the closed ball $B(0,\bar{\varepsilon}) \subset
V_\bo$ invariant.

  Moreover, due to the linear dependence of $\hat{v}$ on $(H,a)$ and the
  estimates \eref{eq:normv} and \eref{eq:ou_H}, we have
  \begin{equation*}
    \fl
    \norm{D_1 F(v,a,u,\lambda) } \le \norm{D_1
    \hat{v}(H,a,u,\lambda) } \cdot \norm{D_1
    H(v,a,u,\lambda) } \le \hat{C}\frac{1}{7KC} \le
    \frac{1}{7K}.
  \end{equation*}
  Thus, $F$ is a contraction on $B(0,\bar{\varepsilon})$ and the
  existence and uniqueness part of the lemma follows immediately from the Banach
  fixed point theorem.

Applying the implicit function theorem at a solution point of
\eref{eq:fixedpoint} provides the smooth dependence of $v$ on $(a,u,\lambda)$.
\end{proof}

\begin{lem}\label{lem:link_v_b}
Choose $K>1$ and $(a,u,\lambda)$ in accordance with lemma~\ref{lem:link_v}.
We define functions
  \begin{equation*}
    \eqalign{
      a_\perp^+(a,u,\lambda) &:= P_s^+(u,\lambda)(\omega^+)
      v^+(a,u,\lambda)(\omega^+),\\
      a_\perp^-(a,u,\lambda) &:= P_u^-(u,\lambda)(-\omega^-)
      v^-(a,u,\lambda)(-\omega^-).
    }
  \end{equation*}
  There are constants $C,\delta^s,\delta^u>0$, that do not depend on $K$, such
  that with $\tilde c$ according to \eref{eq:def_tilde_c}
  \begin{equation}
    \label{eq:est_aperp}
    \eqalign{ 
    \abs{a_\perp^+(a,u,\lambda)} &\le C \norm{a} e^{-\delta^s\omega^+} +
    \frac{\tilde{c}}{K} ,\\ \abs{a_\perp^-(a,u,\lambda)} &\le
    C \norm{a} e^{-\delta^u\omega^-} + \frac{\tilde{c}}{K}
}
  \end{equation}
  and
  \begin{equation}
    \label{eq:est_aperpder}
    \eqalign{
    \abs{D_1 a_\perp^+(a,u,\lambda)} &\le C e^{-\delta^s\omega^+}
    +\frac{1}{K},\\
    \abs{D_1 a_\perp^-(a,u,\lambda)} &\le C
    e^{-\delta^u\omega^-} + \frac{1}{K} .
    }
  \end{equation}
\end{lem}

\begin{proof}
From
  \eref{eq:est_aperp1} and the estimates in the proof of lemma~\ref{lem:link_v} we get
  \begin{equation*}
    \eqalign{
      \abs{a_\perp^+(a,u,\lambda)} &\le C\left(\norm{a}
      e^{-\delta^s\omega^+} + \norm{H(v,u,\lambda)} \right)
\\
      &\le C \norm{a} e^{-\delta^s\omega^+} + \frac{\tilde{c}}{K}.
    }
  \end{equation*}
Next we estimate $D_1 a_\perp^+(a,u,\lambda) = \frac{\partial}{\partial a}
\hat{a}_\perp^+(H(v(a,u,\lambda),u,\lambda),a,u,\lambda)$. Hence
 \begin{equation*}
    \eqalign{
      \abs{D_1 a_\perp^+(a,u,\lambda)} \le& \abs{D_1 \hat{a}_\perp^+(H,a,u,\lambda)
      }\cdot\norm{D_1 H(v,u,\lambda)}\cdot\norm{D_1 v(a,u,\lambda)}\\
      &\qquad +\abs{D_2 \hat{a}^+_\perp(H,a,u,\lambda)}.
    }
  \end{equation*}
Because $v(a,u,\lambda)=\hat{v}(H(v(a,u,\lambda),u,\lambda),a,u,\lambda)$ we get
  \begin{equation*}
    \norm{D_1 v(a,u,\lambda)} \le
    \frac{\norm{D_2\hat{v}(H,a,u,\lambda)}}{1-\norm{D_1
    \hat{v}(H,a,u,\lambda)}\cdot\norm{D_1 H(v,u,\lambda)}}.
  \end{equation*}
  Recall that $\norm{D_2 \hat{v}(H,a,u,\lambda)}\le C$, $\norm{D_1
  \hat{v}(H,a,u,\lambda)}\le C$ and $\norm{D_1
  H(v,u,\lambda)}<\frac{1}{7KC}$, hence
  \begin{equation*}
    \norm{D_1 v(a,u,\lambda)} \le \frac{7KC}{7K-1}.
  \end{equation*}

  Together with \eref{eq:est_aperpder1} and~\eref{eq:est_aperpder2} we finally
  get
  \begin{equation*}
      \abs{D_1 a_\perp^+(a,u,\lambda)} \le 2C \frac{1}{7KC^2} \frac{7KC}{7K-1} + Ce^{-\delta^s\omega^+}
      \le \frac{1}{K} +
      Ce^{-\delta^s\omega^+}.
  \end{equation*}
With similar computations for $a_\perp^-$, the estimates of the lemma follow.
\end{proof}
Recapitulating, we want to note that we find solutions according to
theorem~\ref{thm:linbvp} by inserting the solutions $v(a,u,\lambda)$
into the representation \eref{eq:setting_x}.

The statement of theorem~\ref{thm:linbvp} remains true for `$\omega^-=\infty$'
in the following sense:
\begin{cor}
  \label{cor:linbvp}
Fix $\omega^+$.  There is a constant $c>0$ such that for $\abs{\lambda}<c$,
$u\in U$, $\abs{u}<c$, and $a^+\in \R^n$, $\abs{a^+}<c$, there is a unique pair
of solutions $\left( x^-,x^+ \right)$ of~\eref{eq:system} that satisfy
\begin{description}
  \item[($J$)] $x^{-/+}(a^+,u,\lambda)(0)\in\Sigma$, $x^{-}(a^+,u,\lambda)(0)\in
    W_\lambda^u(\gamma^-)$,\\
  $x^-(a^+,u,\lambda)(0) -
    x^+(a^+,u,\lambda)(0) \in Z$ and
  \item[($B^+$)] $\left( \id -P^+(u,\lambda)(\omega^+) \right) \left(
    x^+(a^+,u,\lambda)(\omega^+) - q^+(u,\lambda)(\omega^+)-a^+
    \right) = 0$.
\end{description}
\end{cor}

\begin{proof}
 Basically the statement follows by setting $\left( \id
 -P^-(u,\lambda)(-\omega^-) \right)a^-=0$ in theorem~\ref{thm:linbvp}: Let
 $\omega^-$ be any value in accordance with theorem~\ref{thm:linbvp}.  Then, due
 to $\left( \id -P^-(u,\lambda)(-\omega^-) \right)a^-=0$, it follows that $\hat
 v^-(-\omega^-)\in \im P^-(u,\lambda)(-\omega^-)=T_{q^-(u,\lambda)(-\omega^-)}W_\lambda^u(\gamma^-)$, cf.
 Lemma~\ref{lem:link_linear}.  Assuming that $W_\lambda^u(\gamma^-)$ is flat around
 $q^-(u,\lambda)(-\omega^-)$, meaning that locally around $q^-(u,\lambda)(-\omega^-)$ the unstable
 manifold $W_\lambda^u(\gamma^-)$ and
 $q^-(u,\lambda)(-\omega^-)+T_{q^-(u,\lambda)(-\omega^-)}W_\lambda^u(\gamma^-)$
 coincide, we find $q^-(u,\lambda)(-\omega^-)+\hat v^-(-\omega^-)\in W_\lambda^u(\gamma^-)$.

Solving fixed point equation \eref{eq:fixedpoint} with that particular $\hat v$,
we find that $x^-(u,\lambda)(-\omega^-)=q^-(u,\lambda)(-\omega^-)+
v^-(-\omega^-)\in W_\lambda^u(\gamma^-)$;
compare also \eref{eq:setting_x}. Hence $x^-$ lies in the unstable manifold of
$\gamma^-$.
\end{proof}

\begin{rem}
  \label{rem:infinity}
 In the same sense theorem~\ref{thm:linbvp} remains true for `$\omega^+=\infty$'.
\end{rem}

\subsection{The jump function for a short Lin orbit segment}
\label{sec:jump_2}

According to theorem~\ref{thm:linbvp}, for given $a = (a^-,a^+)$, $u$, $\lambda$
and $\bo=(\omega^-,\omega^+)$, there is a unique short Lin orbit segment
$X=(x^-,x^+)$.  Note that $X$ depends in particular on $\bo$, which is not
reflected in our notation so far.  To emphasize this dependence from we now
use the notation $X_\bo$, and similarly $x^\pm_\bo$ and $v^\pm_\bo$.  We define
the jump function $\Xi$ as
\begin{equation}\label{eq:def_Xi}
  \Xi(\bo,a,u,\lambda) := x^-_\bo (a,u,\lambda)(0) -
  x^+_\bo(a,u,\lambda)(0).
\end{equation}
Using that 
$x^{-/+}_\bo(a,u,\lambda)(t) = q^{-/+}(u,\lambda)(t) +
v^{-/+}_\bo(a,u,\lambda)(t)$, we can write $\Xi$ in the form
\begin{equation}\label{eq:jumpsplit}
 \Xi(\bo,a,u,\lambda)=
  \xi^\infty(u,\lambda) + \xi(\bo,a,u,\lambda),
\end{equation}
where 
\begin{equation}\label{eq:def_xi}
\begin {array}{l}
 \xi^\infty(u,\lambda):=q^-(u,\lambda)(0) - q^+(u,\lambda)(0),
\\[1ex]
\xi(\bo,a,u,\lambda) := v^-_\bo(a,u,\lambda)(0) -
v^+_\bo(a,u,\lambda)(0). 
\end{array}
\end{equation}
Recall that $(v^-_\bo,v^+_\bo)$ is the solution of the fixed point equation
\eref{eq:fixedpoint} and, hence, solves the boundary value problem
(\eref{eq:nonlinvar}, ($J_v$), ($B_v^-$), ($B_v^+$)).

The term $\xi^\infty$ reflects the intersection of the stable and unstable
manifolds of $\gamma^-$ and $\gamma^+$ respectively. We present examples for
suitable choices of $\xi^\infty$ in section~\ref{sec:applications}. Here we
focus on estimates of $\xi$.

In order to establish those estimates, we impose some assumptions on the leading
eigenvalues of $\gamma^-$ and $\gamma^+$.  Let $\mu^-_s$ denote the leading
stable eigenvalue or the leading stable Floquet exponent
of $\gamma^-$ depending on whether $\gamma^-$ is an equilibrium point or a
periodic orbit with minimal period $T^->0$.  Similarly, let $\mu^+_u$ be the
leading unstable eigenvalue or the leading unstable Floquet exponent of
$\gamma^+$.  We assume the following:
\begin{hyp}
  \label{hyp:leading_ev}
  $\mu^-_s,\mu^+_u$ {\rm are real and simple}.
\end{hyp}
 
Further, for the sake of simplicity, we also assume
\begin{hyp}\label{hyp:dim_Z}
 $\dim Z=1$.
\end{hyp}
Let $Z=\mbox{span}\, \{z\}$, $|z|=1$.Then, since $\xi\in Z$,
\[
 \xi(\bo,a,u,\lambda)=\langle z,\xi(\bo,a,u,\lambda)\rangle z.
\]
Further we assume
\begin{hyp}
  \label{hyp:orth_decomp}
  {\rm The direct sum decomposition~\eref{eq:Y} is orthogonal with respect to the
  used scalar product $\langle\cdot,\cdot\rangle$.}
\end{hyp}

\begin{lem}\label{lem:jump_2}
 Let $a,u,\lambda,\bo$ be in accordance with theorem~\ref{thm:linbvp}, and let
 hypotheses~\ref{hyp:leading_ev}--\ref{hyp:orth_decomp} hold.  Then, \[
 \xi(\bo,a,u,\lambda)={\mathcal O}(\|a\|).  \] The ${\mathcal O}(\cdot)$ limit
 holds for $\|a\|\to 0$ uniformly in $u,\lambda,\bo$.
\end{lem}

\begin{proof}
 According to the definition of $\xi$ and hypothesis~\ref{hyp:orth_decomp} we
 find that
\begin{equation*}
  \eqalign{
  \langle z,\xi(\bo,a,u,\lambda) \rangle &= \langle z,
  v^-_\bo(a,u,\lambda)(0) - v^+_\bo(a,u,\lambda)(0) \rangle\\
  &= \langle z, (\id - P^-(u,\lambda)(0))v^-_\bo(a,u,\lambda)(0) \rangle \\
  &\qquad- 
\langle z, (\id - P^+(u,\lambda)(0))v^+_\bo(a,u,\lambda)(0) \rangle.
  }
\end{equation*}
Since $v$ satisfies the fixed point equation \eref{eq:fixedpoint}, according
to~\eref{eq:linop} we find that
\begin{equation}
  \label{eq:jump_2}
  \fl\eqalign{
  \langle z,\xi(\bo,a,u,\lambda) \rangle &= 
\langle\Phi^-(0,-\omega^-)^T z, (\id - P^-(u,\lambda)(-\omega^-))a^- \rangle 
\\
&\quad - \langle \Phi^+(0,\omega^+)^T z, (\id - P^+(u,\lambda)(\omega^+))a^+ \rangle
\\
  &\quad+ \langle z,\int_{-\omega^-}^0 \Phi^-(0,\tau)\left( \id - P^-(\tau)
  \right) h^-(\tau,v^-_\bo(\tau),u,\lambda)
    d\tau \rangle
\\
  &\quad+ \langle z,\int_0^{\omega^+} \Phi^+(0,\tau)\left( \id - P^+(\tau) \right)
  h^+(\tau,v^+_\bo(\tau),u,\lambda) d\tau \rangle.
  }
\end{equation}
First note that 
\[ \fl \langle \Phi^+(0,\omega^+)^T z, (\id -
P^+(u,\lambda)(\omega^+))a^+ \rangle=\langle \Phi^+(0,\omega^+)^T (\id -
P^+(0))^T z, a^+ \rangle.  \] 
Further, $\Phi^+(0,\cdot)^T$ is a solution of the
adjoint of the variational equation along $\gamma^+$.  Exponential
dichotomy/trichotomy of this equation yields that, uniformly in
$u,\lambda,\omega^{+}$,
\[ 
\langle \Phi^+(0,\omega^+)^T(\id - P^+(0))^T  z, a^+
\rangle= {\mathcal O}(|a^+|).  
\] 
Similar arguments apply to
$\langle\Phi^-(0,-\omega^-)^T z,(\id - P^-(u,\lambda)(-\omega^-)) a^- \rangle$.

Standard results from Lin's method (cf.~\cite{Knobloch2004, Riess2008,
Sandstede1993}) imply that the integral terms in \eref{eq:jump_2} are also
${\mathcal O}(\abs{a^-})$ or  ${\mathcal O}(\abs{a^+})$ uniformly in
$u,\lambda,\bo$, respectively. Note that the arguments in \cite{Knobloch2004,
Sandstede1993}, where $\gamma^\pm$ are always equilibria, apply also in the
present situation. These arguments are mainly based on the exponential dichotomy
of the variational equation along $\gamma^\pm$ and the structure of $h$.
\end{proof}

\begin{cor}
  \label{cor:jump_2}
 Let $a^+,u,\lambda,\omega^+$ be in accordance with corollary~\ref{cor:linbvp},
 and let hypotheses~\ref{hyp:leading_ev}--\ref{hyp:orth_decomp} hold.
 Then, 
 \[ 
 \fl\langle z,\xi(\omega^+,a^+,u,\lambda)\rangle=- \langle
 \Phi^+(0,\omega^+)^T z, (\id - P^+(u,\lambda)(\omega^+))a^+
 \rangle+o(\abs{a^+}).  \]
 The $o(\cdot)$ limit holds for $\abs{a^+}\to 0$
 uniformly in $u,\lambda,\omega^+$.
 \end{cor}

\begin{proof}
 As in the proof of corollary~\ref{cor:linbvp}, we set $a^-=0$. Then estimates in
 \cite[lemma~3.20]{Sandstede1993} provide the the corresponding estimate of the
 integral terms in \eref{eq:jump_2}.
\end{proof}

\section{Joining two short Lin orbit segments}
\label{sec:twolinks}

Let $\gamma_l\cup q_l\cup \gamma$ and $\gamma\cup q_r\cup\gamma_r$ be
consecutive short heteroclinic chain segments. The objective of this section is
to join the related short Lin orbit segments $X_l=(x_l^-,x_l^+)$ and
$X_r=(x_r^-,x_r^+)$ to a long Lin orbit segment. Here we focus on the case where
$\gamma$ is a hyperbolic periodic orbit with minimal period $T>0$.  We use the
same notation as in \sref{sec:onelink} with an additional subscript `$l$' or
`$r$' referring to the left short Lin orbit segments $X_l$ or right short Lin
orbit segments $X_r$, respectively.  However, for convenience we use the short
notation $\omega^- = \omega_l^-$ and $\omega^+ = \omega_r^+$.

In the construction the transition time $\tau$ from $\Sigma_l$ to $\Sigma_r$
plays a major role. In the present context, $\tau$ is directly related to the
number $\nu$ of rotations of the long Lin orbit segment along $\gamma$.

\begin{thm}\label{thm:lhc}
  Fix $\omega^-,\omega^+>0$.
  There are constants $c, N>0$ such that for all $\abs{\lambda}<c$,
  $u = (u_l,u_r)$, $\norm{u}<c$, $a^-,a^+\in\R^n$, $\abs{a_l^-},\abs{a_r^+}<c$, and for all $\nu \in \N \cap (N,\infty)$,
  there is a transition time $\tau$ and a unique triple $x =
  (x_l,x_m,x_r)$, $x(\cdot) = x(\nu,a_l^-,a_r^+,u,\lambda)(\cdot)$, of solutions
  of~\eref{eq:system} that satisfy
\begin{description}
    \item[($J$)]
      $x_{l}(0), x_{m}(0) \in\Sigma_l$,\;  $x_m(\tau), x_{r}(0) \in\Sigma_r$,
\; $x_l(0) - x_m(0) \in Z_l$,\; $x_m(\tau) - x_r(0) \in Z_r$,
    \item[($B_l$)] $\left( \id - P_l^-(u_l,\lambda)(-\omega^-) \right) \left(
      x_l(-\omega^-) -
      q_l^-(u_l,\lambda)(-\omega^-)-a^- \right) = 0$,
    \item[($B_r$)]
      $\left( \id -
      P_r^+(u_r,\lambda)(\omega^+) \right) \left(
      x_r(\omega^+) - q_r^+(u_r,\lambda)(\omega^+)-a^+
      \right) = 0$.
\end{description}
\end{thm}
Figure~\ref{fig:lhc} visualizes the statement of the theorem.
\begin{figure}[ht]
\begin{center}
\includegraphics[scale=.8]{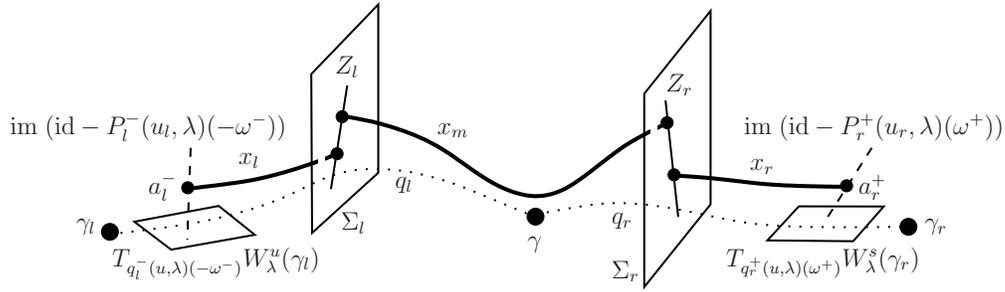}
\caption{\label{fig:lhc}Sketch of a long Lin orbit segment $(x_l,x_m,x_r)$ near a long heteroclinic
chain $\gamma_l \cup q_l \cup \gamma \cup q_r \cup \gamma_r$.}
\end{center}
\end{figure}

We perform the proof of theorem~\ref{thm:lhc} in two steps, see
sections~\ref{sec:flowneargamma} and~\ref{sec:couplinginsigma}.  First we study the
flow along $\gamma$ by means of a Poincar\'e map
$\Pi:\Sigma_\gamma\to\Sigma_\gamma$, where $\Sigma_\gamma$ is an appropriate
Poincar\'e section.  More precisely, we show that there are $\Pi$-orbit segments
$y$ satisfying certain boundary conditions in $\Sigma_\gamma$. To that end, we
employ a similar technique as used in the theory of Lin's method for discrete
dynamical systems \cite{Knobloch2004}.  We denote the $f$-orbit that
is the suspension of the $\Pi$-orbit $y$ by $x(y)$.  Then we couple $x_l^+$ and $x(y)$
and simultaneously $x(y)$ and $x_r^-$. The partial orbit $x_m$ is composed
of $x_l^+$, $x(y)$ and $x_r^-$.  The (in this
context prescribed) times $\omega^+_l$ and $\omega^-_r$ and the duration of
$x(y)$ add up to the transition time $\tau$.  Further, we have $x_l = x_l^-$ and
$x_r = x_r^+$.

The statement of theorem~\ref{thm:lhc} remains true for $\omega^-=\omega^+=\infty$ in the following sense:
\begin{cor}\label{cor:lhc}
 There are constants $c, N>0$ such that for all $\abs{\lambda}<c$,
  $u = (u_l,u_r)$, $\abs{u_l},\abs{u_r}<c$,  and for all $\nu \in \N \cap (N,\infty)$,
  there is a unique triple $x =
  (x_l,x_m,x_r)$, $x(\cdot) = x(\nu,u,\lambda)(\cdot)$, of solutions
  of~\eref{eq:system} such that for some transition time $\tau$
\begin{description}
    \item[($J$)]
      $x_{l}(0), x_{m}(0) \in\Sigma_l$,\;  $x_m(\tau), x_{r}(0) \in\Sigma_r$,\;
      $x_l(0) - x_m(0) \in Z_l$,\; $x_m(\tau) - x_r(0) \in Z_r$,
\item[($B$)]
$x_{l}(0)\in W^u(\gamma_l)$,\quad $x_{r}(0)\in W^s(\gamma_r)$.
\end{description}
For fixed $\nu$ the solution $x$ depends smoothly on $(u,\lambda)$.
\end{cor}
The proof of corollary~\ref{cor:lhc} will be given at the end of
section~\ref{sec:couplinginsigma}.

\begin{defn}
If $\omega^-,\omega^+<\infty$ we call the triple $(x_l,x_m,x_r)$ a long Lin
orbit segment, and in case $\omega^-=\omega^+=\infty$ we call $(x_l,x_m,x_r)$ a
heteroclinic Lin orbit connecting $\gamma_l$ and $\gamma_r$.  If $\gamma_l
\equiv \gamma_r$, we call $(x_l,x_m,x_r)$ a homoclinic Lin orbit connecting
$\gamma_l$ to itself.
\end{defn}

\subsection{The flow near $\gamma$}
\label{sec:flowneargamma}

Let $\Sigma_\gamma$ be a Poincar\'e section of $\gamma$.
We consider the discrete dynamical system defined by the Poincar\'e map
$\Pi:\Sigma_\gamma\times\R^m\to\Sigma_\gamma$ near $p_\gamma:=\gamma\cap\Sigma_\gamma$:
\begin{equation}
\label{eq:poincare}
  y(n+1) = \Pi(y(n),\lambda).
\end{equation}

The intersection points of $q_l^+(u_l,\lambda)$ and $q_r^-(u_r,\lambda)$ with
the Poincar\'e section $\Sigma_\gamma$ define solutions $q_d^+(u_l,\lambda)(n)$,
$q_d^-(u_r,\lambda)(n)$ of~\eref{eq:poincare} lying in the stable/unstable
manifold of the hyperbolic $\Pi$-equilibrium $p_\gamma$.  
Let $Y_\gamma$ be the $(n-1)$-dimensional subspace of $\R^n$ such that
\[
\Sigma_\gamma= p_\gamma+Y_\gamma.
\]
The
variational equation along $q_d^\pm$ has an exponential dichotomy on
$\Z^\pm$ and we denote the corresponding projections by $Q^+(u_l,\lambda)$ and
$Q^-(u_r,\lambda)$.
Note again that the images of $Q^\pm$ are well-determined:
\begin{equation}\label{eq:def_imQ}
\begin{array}{l}
  \im
  Q^+(u_l,\lambda)(0)=T_{q_d^+(u_l,\lambda)(0)}W^s_{\Pi,\lambda}(p_\gamma)\subset
  T_{q_d^+(u_l,\lambda)(0)}W_\lambda^s(\gamma),
\\[1ex]
\im
Q^-(u_r,\lambda)(0)=T_{q_d^-(u_r,\lambda)(0)}W^u_{\Pi,\lambda}(p_\gamma)\subset
T_{q_d^-(u_r,\lambda)(0)}W_\lambda^u(\gamma),
\end{array}
\end{equation}
where $W^{s(u)}_{\Pi,\lambda}$ denotes the (un)stable manifold of the mapping
$\Pi=\Pi(\lambda)$ and we use the short notation $W^{s(u)}_\Pi$ at $\lambda=0$.
However, there is some freedom in choosing the kernels of $Q^\pm$, which
allows us to use the ideas from Lin's method for discrete dynamical systems in the
following, and also allows us to couple the solution of the discrete system with the
solutions of the continuous system, cf.~\eref{eq:projections}
and~\eref{eq:imQ_subset} below.

\begin{lem}\label{lem:lambdalemma0}
  There are constants $\tilde{c}, \bar{c}>0$ and $N\in\N$ such that for all $\abs{\lambda}<\bar{c}$, $u =
  (u_l,u_r)$, $u_l\in U_l$, $u_r\in U_r$, $\abs{u_l},\abs{u_r}<\bar{c}$, $\nu>N$
  and $b=(b^+,b^-)\in Y_\gamma\times Y_\gamma$, $\norm{b}<\tilde{c}$, there is a unique solution $y =
  y(b,u_l,u_r,\lambda)$ of~\eref{eq:poincare} that satisfies
  \begin{description}
    \item[($\mathcal B$)] $Q^+(u_l,\lambda)(0) \left(y(b,u,\lambda)(0) -
      q_d^+(u_l,\lambda)(0)-b^+\right) =
      0$,
    \\
$Q^-(u_r,\lambda)(0) \left(y(b,u,\lambda)(\nu) -
      q_d^-(u_r,\lambda)(0)- b^-\right)
      = 0$.
  \end{description}
\end{lem}
Again we suppress the dependence of $y$ on $\nu$ from our notation.  Note that
this lemma is a discrete version of the existence and uniqueness
result on Shilnikov data, cf. \cite{Deng1989}.  However, also in view of its
application in the following section~\ref{coup_in_Sigma}, we consider a
reformulation by using small perturbations of $q_d^+(n)$ and $q_d^-(n)$, similar
to \sref{sec:onelink}: For given $\nu\in\N$ we define
$\nu^+:=\lfloor\frac{\nu}{2}\rfloor$, the integer part of $\nu$, and
$\nu^-:=\nu-\nu^+$.  Further, let us think of $y$ as being composed of two
partial orbits as follows
\begin{equation}\label{eq:def_y_pm_a}
 y(n)=\left\{
\begin{array}{ll}
 y^+(n),&n\in[0,\nu^+]\cap\N
\\
y^-(n-\nu),&n\in[\nu^+,\nu]\cap\N
\end{array}
\right.
\end{equation}
with the additional demand that
\begin{equation}\label{eq:def_y_pm_b}
 y^+(\nu^+)=y^-(-\nu^-).
\end{equation}
We write
\begin{equation}\label{eq:def_y_pm_c}
  y^+(n) = q_d^+(u_l,\lambda)(n) + w^+(n)\mbox{ and } y^-(n) =
  q_d^-(u_r,\lambda)(n) + w^-(n).
\end{equation}
If $y^\pm$ solve \eref{eq:poincare}, then $w^\pm(\cdot)$ satisfy the following difference equations:
\begin{equation}
\label{eq:dnonlinvar}
\eqalign{
w^-(n+1) &= D_1
\Pi(q_d^{-}(u_r,\lambda)(n),\lambda) w^-(n) +
h^-(n,w^-,u_r,\lambda),\\
w^+(n+1) &= D_1
\Pi(q_d^{+}(u_l,\lambda)(n),\lambda) w^+(n) +
h^+(n,w^+,u_l,\lambda),
}
\end{equation}
where
\begin{eqnarray*}
h^\pm(n,w,u,\lambda) :=&
\Pi(q_d^\pm(u,\lambda)(n)+w,\lambda) -
\Pi(q_d^\pm(u,\lambda)(n),\lambda)\\
&\quad - D_1
\Pi(q_d^\pm(u,\lambda)(n),\lambda)w.
\end{eqnarray*}

For $\nu\in\N$ let $S_\nu$ denote the space of functions $\{ 0,\dots,\nu \}\to Y_\gamma$, and let 
$S_{-\nu}$ denote the space of functions $\{-\nu,\dots,0 \}\to Y_\gamma$.
For given $\nu^+$ and $\nu^-$ we write $\bn:=(\nu^+,\nu^-)$, and we
define the space 
\begin{equation*}
  W_\bn := S_{\nu^+} \times S_{-\nu^-}.
\end{equation*}
Then lemma~\ref{lem:lambdalemma0} follows from
\begin{lem}
\label{lem:lambdalemma}
There are  constants $\bar{\epsilon},\tilde\mathfrak{c},\bar\mathfrak{c}$ and
$N\in\N$ such that for $\abs{\lambda}\leq\bar\mathfrak{c}$, $u:=(u_l,u_r)\in
U_l\times U_r$, with $\norm{u}\leq\bar\mathfrak{c}$, $\nu>N$ and $b=(b^+,b^-)\in
Y_\gamma\times Y_\gamma$, $\norm{b}\leq\tilde\mathfrak{c}$, there is a unique
pair $w_\bn = (w^+_\bn,w^-_\bn) \in W_\bn$, $w_\bn = w_\bn(b,u,\lambda)$, of
solutions of~\eref{eq:dnonlinvar} in an $\bar{\epsilon}$-neighborhood of $0\in
W_\bn$ such that
 \begin{description}
    \item[(${\mathcal B}_w$)] $Q^+(u_l,\lambda)(0)( w^+_\bn(b,u,\lambda)(0)-b^+) = 0$, \quad
      $Q^-(u_r,\lambda)(0)( w^-_\bn(b,u,\lambda)(0)-b^-)= 0$,
    \item[($\mathcal C$)] $w^-_\bn(b,u,\lambda)(-\nu^-) -
      w^+_\bn(b,u,\lambda)(\nu^+) =
      q_d^+(u_l,\lambda)(\nu^+) - q_d^-(u_r,\lambda)(-\nu^-)$.
  \end{description}
\end{lem}

\begin{proof}
 To some extent the arguments run parallel to those used in
 section~\ref{sec:exptrich}. Here we only give a sketch of the proof;
 for more details we refer to~\cite{Riess2008}.

First we consider the inhomogeneous equations 
\begin{equation}
\label{eq:dlinvar_inh}
\eqalign{
w^-(n+1) &= D_1
\Pi(q_d^{-}(u_r,\lambda)(n),\lambda) w^-(n) + g^-(n),
\\
w^+(n+1) &= D_1
\Pi(q_d^{+}(u_l,\lambda)(n),\lambda) w^+(n) +g^+(n),
}
\end{equation}
with boundary conditions
\begin{description}
 \item[(${\mathcal B}_w$)] $Q^+(u_l,\lambda)(0)( w^+(0) - b^+)=0$,\quad
   $Q^-(u_r,\lambda)(0)( w^-(0) - b^-)=0$, 
 \item[(${\mathcal B}_\beta$)]
   $(\id-Q^+(u_l,\lambda)(\nu^+)) w^+(\nu^+) = \beta^+$,\quad
   $(\id-Q^-(u_r,\lambda)(-\nu^-)) w^-(-\nu^-) = \beta^-$,
\end{description}
for given $\beta^+\in\im\,(\id-Q^+(u_l,\lambda)(\nu^+))$ and
$\beta^-\in\im\,(\id-Q^-(u_r,\lambda)(-\nu^-))$.  We  write
$\beta:=(\beta^+,\beta^-)$.  Similar to the proof of
lemma~\ref{lem:link_linear},
we find that the boundary value problem
(\eref{eq:dlinvar_inh},($B_b$),($B_\beta$)) has a unique solution $\bar w\in
W_\bn$, $\bar w=\bar w(g,b,\beta,u,\lambda)$.

Next we replace the boundary condition (${\mathcal B}_\beta$) by 
\begin{description}
 \item[(${\mathcal B}_d$)] $w^+(\nu^+)-w^-(-\nu^-)=d$, \quad $d\in Y_\gamma$.
\end{description}
Indeed there is a $\beta=\beta(d)$ such that $\hat w(g,b,d,u,\lambda):=\bar
w(g,b,\beta(d),u,\lambda)$ is the unique solution of the boundary value problem
(\eref{eq:dlinvar_inh},($B_b$),($B_d$)). The argument for this fact runs parallel to
the corresponding construction in \cite{Sandstede1993} or \cite{Knobloch2004}.

Further, similar to the proof of lemma~\ref{lem:link_v}, we consider a fixed
point equation whose solutions also satisfy the coupling condition (${\mathcal
C}$): For that we define
\begin{equation}\label{eq:def_d}
 d_\bn(u,\lambda):=q_d^+(u_l,\lambda)(\nu^+)-q_d^-(u_r,\lambda)(-\nu^-).
\end{equation}
Finally, we consider the fixed point equation
\begin{equation}\label{eq:fixedpoint_d}
 w=\hat w({\mathcal H}(w,u,\lambda),b,d_\bn,u,\lambda).
\end{equation}
Here, ${\mathcal H}$ is the discrete pendant of the Nemytskii operator defined in
\eref{eq:nem_op_H}.  Similar to the procedure in section~\ref{sec:exptrich} one
proves that \eref{eq:fixedpoint_d} has a unique fixed point.
\end{proof}
We define functions $B_{Y_\gamma\times Y_\gamma}(0,\tilde\mathfrak{c})\times
B_{U_l\times U_r}(0,\bar\mathfrak{c})\times B_{\R^m}(0,\bar\mathfrak{c})\to
B_{Y_\gamma}(0,\bar\epsilon)$
  \begin{equation}\label{eq:def_b_perp}
  \eqalign{
    b_\perp^+(b,u,\lambda) &:= \left( \id -
    Q^+(u_l,\lambda)(0) \right) w^+_\bn(b,u,\lambda)(0),\\
    b_\perp^-(b,u,\lambda) &:= \left( \id -
    Q^-(u_r,\lambda)(0) \right) w^-_\bn(b,u,\lambda)(0).
  }
  \end{equation}
First note that $b_\perp^\pm$ depend smoothly on $(b,u,\lambda)$, and note
further that $b_\perp^\pm$ depend also on $\bn$.  The `size' of these functions
is closely related to the `size' of the jumps $\xi_l$ and $\xi_r$.

In accordance with hypothesis~\ref{hyp:leading_ev} we assume
\begin{hyp}\label{hyp:ev_p_gamma}
{\rm  The leading
stable and unstable eigenvalues $\mu^{s/u}$ of $p_\gamma$ are real and simple. }
\end{hyp}
Note that here $\mu^{s/u}$ denote the eigenvalues of $p_\gamma$ (and not the
Floquet exponents of $\gamma$ as in section~\ref{sec:onelink}).
\begin{hyp}\label{hyp:approach_q}
{\rm $q_d^+(u_l,\lambda)(\cdot)$ and $q_d^-(u_r,\lambda)(\cdot)$ approach
$p_\gamma$ along the leading stable and unstable direction, respectively.}
\end{hyp}

\begin{hyp}\label{hyp:approach_b_perp}
 {\rm $b_\perp^+(b,u,\lambda)$ and $b_\perp^-(b,u,\lambda)$ are not in the
 strong stable subspace of the adjoint of the
 variational equation along $q^+(u_l,\lambda)(\cdot)$ and
 $q^-(u_r,\lambda)(\cdot)$, respectively, for $n=0$.}
\end{hyp}

\begin{lem}\label{lem:lambdalemma_b}
Assume hypotheses~\ref{hyp:ev_p_gamma}--\ref{hyp:approach_b_perp}. Further,
  let the assumptions of lemma~\ref{lem:lambdalemma} hold.  
  There are functions $c^{s/u} = c^{s/u}(b,u,\lambda)$ such that
  \begin{equation}
    \label{eq:est_bperp}
  \eqalign{
    \abs{ b_\perp^+(b,u,\lambda) } &= c^u(b,u,\lambda) \left( \mu^u \right)^{-\nu} +
    o( | \mu^u |^{-\nu} ),\\
    \abs{ b_\perp^-(b,u,\lambda) } &= c^s(b,u,\lambda) \left( \mu^s \right)^{\nu} +
    o( | \mu^s |^{\nu} ).
  }
  \end{equation}
  There is a constant $C>0$ such that
  \begin{equation}
    \label{eq:est_bperpder}
    \abs{D_1 b_\perp^+(b,u,\lambda) } \le C | \mu^u |^{-\nu},\qquad
    \abs{D_1 b_\perp^-(b,u,\lambda) } \le C | \mu^s |^{\nu}.
   \end{equation}
  The functions $c^{u/s}$ are smooth and $c^{u/s}(0,0,0)\not=0$.  The $o(\cdot)$-terms are valid for $\nu$ tending to infinity.
\end{lem}
The estimates in~\cite{Deng1989} (applied to discrete systems) already provide
that $\abs{b_\perp^+(b,u,\lambda)} = \mathcal{O}( | \mu^u |^{-\nu} )$, cf. also corollary~\ref{cor:lambdalemma_b}, but
they do not give information about the leading term.  However, this information
is important for the jump estimates and consequently for the construction of
bifurcation equations.  Note that the information about leading terms of the
derivatives in~\eref{eq:est_bperpder} is not needed for our purposes here, but
can be computed in a similar manner as in~\cite{Knobloch2004, Sandstede1993}.

\begin{proof}
With $\hat{b}_\perp^+=\hat{b}_\perp^+(b,u,\lambda):= b_\perp^+(b,u,\lambda)/|b_\perp^+(b,u,\lambda)|$
we can write
  \begin{equation}
    \label{eq:bscalarproduct}
   \abs{b_\perp^+(b,u,\lambda)} =\left\langle \hat{b}_\perp^+,
    b_\perp^+(b,u,\lambda) \right\rangle.
  \end{equation}

Note that $(w^-_\bn,w^+_\bn)$ solves \eref{eq:dnonlinvar}. So, applying the
variation of constants formula to \eref{eq:dlinvar_inh} and replacing there
$g^\pm$ by $h^\pm$ finally provides
\begin{equation*}
    \fl\eqalign{
    b_\perp^+(b,u,\lambda) =& \Psi^+(0,\nu^+) \beta^+ \\
    &- \underbrace{\sum_{m=1}^{\nu^+}
    \Psi^+(0,m) (\id - Q^+(0)) h^+(m-1,w^+(m-1),u_l,\lambda)}.\\
    &\hspace*{5cm} =: \mathcal{S}
    }
  \end{equation*}
Here $\Psi^+(\cdot,\cdot)$ is the transition matrix of the homogeneous equation
of \eref{eq:dlinvar_inh}. Note that $\Psi^+$ depends on $u_l$ and $\lambda$.
Further, $\beta^+$ is defined by the boundary condition ($\mathcal{B}_\beta$).

  Replacing $b_\perp^+$ in the scalar product~\eref{eq:bscalarproduct} yields
   \begin{equation}
    \label{eq:bscalarproduct2}
    \eqalign{
    \langle \hat{b}_\perp^+, b_\perp^+(b,u,\lambda) \rangle  &=
    \langle \hat{b}_\perp^+, \Psi^+(0,\nu^+) \beta^+ \rangle -\langle \hat{b}_\perp^+, \mathcal{S}
    \rangle\\
    &= \langle \Psi^+(0,\nu^+)^T (\id - Q^+(0))^T \hat{b}_\perp^+, \beta^+ \rangle - \langle \hat{b}_\perp^+, \mathcal{S}
    \rangle.
    }
  \end{equation}
Considerations similar to those in \cite{Knobloch2004,Sandstede1993} show that
the leading-order term of $b_\perp^+(b,u,\lambda)$ will be determined by
$\langle \hat{b}_\perp^+, \Psi^+(0,\nu^+) \beta^+ \rangle$ or $\langle
\Psi^+(0,\nu^+)^T (\id - Q^+(0))^T \hat{b}_\perp^+, \beta^+ \rangle$,
respectively.  Note in this respect that, due to the coupling condition
($\mathcal{B}_\beta$), the quantity $\beta^+$ depends on $\bn$.

Computations in~\cite{Knobloch2004,Sandstede1993} show that under
hypothesis~\ref{hyp:approach_b_perp}
  \begin{equation}\label{eq:est_b}
    \Psi^+(0,\nu^+)^T (\id - Q^+(0))^T \hat{b}_\perp^+ = \eta^+(b,u,\lambda) \left( \mu^u \right)^{-\nu^+} +
    o( | \mu^u |^{-\nu^+} ),
  \end{equation}
  where $\eta^+(b,u,\lambda)\not=0$ is a certain eigenvector of
  $(D_1\Pi(p_\lambda,\lambda)^{-1})^T$ belonging to $(\mu^u)^{-1}$.

Next we consider $\beta^+$.  Combining ($\mathcal{B}_\beta$), ($\mathcal{B}_d$)
and \eref{eq:def_d} yields
  \begin{equation}
    \label{eq:betapm}
    \eqalign{
    \beta^+ - \beta^- &= q_d^-(u_r,\lambda)(-\nu^-) -
    q_d^+(u_l,\lambda)(\nu^+)\\
    &\quad\; - Q^+(\nu^+)w^+(\nu^+) + Q^-(-\nu^-)w^-(-\nu^-).
    }
  \end{equation}
  We define projections $\tilde{Q}(\nu)=\tilde{Q}(u,\lambda)(\nu)$ by,
  cf.~\cite{Knobloch2004} for their existence, 
\[
 \im\,\tilde{Q}(\nu) =\im (\id - Q^+(\nu^+))
\quad {\rm and}\quad 
\ker\, \tilde{Q}(\nu) = \im (\id -Q^-(-\nu^-)).
\]
Applying $\tilde{Q}(\nu)$ to~\eref{eq:betapm} 
yields
\begin{equation}\label{eq:rep_beta}
 \eqalign{
    \beta^+  &=\tilde{Q}(\nu)(q_d^-(u_r,\lambda)(-\nu^-) -
    q_d^+(u_l,\lambda)(\nu^+))\\
    &\quad\; -\tilde{Q}(\nu) Q^+(\nu^+)w^+(\nu^+) +\tilde{Q}(\nu) Q^-(-\nu^-)w^-(-\nu^-).
    }
\end{equation}
In \cite{Knobloch2004} it has been shown that the leading-order term of the
right-hand side of \eref{eq:rep_beta} is determined by the first addend, and
estimates given there reveal that under hypothesis~\ref{hyp:approach_q} we have
that
  \begin{equation}\label{eq:est_beta}
   \beta^+ = \eta^u(u,\lambda) (\mu^u)^{-\nu^-} +
    o( | \mu^u |^{-\nu^-} ),
  \end{equation}
  where $\eta^u(u_l,\lambda)\not=0$ is a certain eigenvector of
  $D_1\Pi(p_\lambda,\lambda)$ belonging to $\mu^u$.

Combining \eref{eq:bscalarproduct2}, \eref{eq:est_b} and \eref{eq:est_beta} provides
\begin{equation*}
 \langle \hat{b}_\perp^+, b_\perp^+(b,u,\lambda)\rangle=\langle \eta^+(b,u,\lambda),\eta^u(u,\lambda)\rangle (\mu^u)^{-\nu}
+o( | \mu^u |^{-\nu^-}).
\end{equation*}
Here we also used that $\langle \hat{b}_\perp^+, \mathcal{S} \rangle=o( | \mu^u
|^{-\nu} )$;
  we refer again to~\cite{Knobloch2004, Sandstede1993} for details of
  the necessary computations.  

Finally, from the definition of $\eta^+$ and $\eta^u$ it follows that 
\begin{equation}\label{eq:def_c^u}
c^u(b,u,\lambda) := \langle \eta^+(b,u,\lambda), \eta^u(u,\lambda) \rangle\not=0. 
\end{equation}
More precisely, linear algebra shows that $\eta^+$ and $\eta^u$ cannot be orthogonal.
Further, due to hypotheses~\ref{hyp:approach_q} and \ref{hyp:approach_b_perp}
both $\eta^+$ and $\eta^u$ are different from zero. 

Similar computations yield the statement on $\abs{b_\perp^-}$.

The smoothness of $c^{u/s}$ follows from the smoothness of $\eta^+$ and $\eta^u$.
The estimates of the derivatives follow immediately from the considerations
  in~\cite{Deng1989}.
\end{proof}

The following is an immediate consequence of lemma~\ref{lem:lambdalemma_b}:
\begin{cor}\label{cor:lambdalemma_b}
 Let the assumptions of lemma~\ref{lem:lambdalemma_b} hold. Then there is a
 constant $C$ such that for all $(b,u,\lambda)\in B_{Y_\gamma\times
 Y_\gamma}(0,\tilde\mathfrak{c})\times B_{U_l\times
 U_r}(0,\bar\mathfrak{c})\times B_{\R^m}(0,\bar\mathfrak{c})$ 
 \[
 \abs{b_\perp^+(b,u,\lambda)}<C(\mu^u)^{-\nu}, \quad
 \abs{b_\perp^-(b,u,\lambda)}<C(\mu^s)^{-\nu}.  
 \]
\end{cor}

\subsection{The coupling within $\Sigma_\gamma$}\label{coup_in_Sigma}
\label{sec:couplinginsigma}

Now we have all the ingredients to couple two pairs of solutions $(x_l^-,x_l^+)$
and $(x_r^-,x_r^+)$ of the continuous system with a solution
$y$ of the discrete system, effectively combining the
solutions $x_l^+$, $x_r^-$ and $y$ into one solution.

We fix the times $\omega_l^+$ and $\omega_r^-$ at sufficiently large values by
using a fixed Poincar\'e section $\Sigma_\gamma$ and then switch to the discrete
dynamical system to describe the dynamics near the periodic orbit.  To reflect
the nature of this setting we rename $\omega_l^+$ as $\Omega^+$ and $\omega_r^-$
as $\Omega^-$.  We choose $\Omega^\pm$ such that $q^+(u_l,\lambda)(\Omega^+),
q^-(u_r,\lambda)(-\Omega^-) \in \Sigma_\gamma$.

In our analysis we consider $x_{l/r}^\pm$ as perturbations of $q_{l/r}^\pm$: 
\[
x_{l/r}^\pm=q_{l/r}^\pm(u_{l/r},\lambda)+v_{l/r}^\pm(u_{l/r},\lambda).  
\]
Further, we represent, as in the previous section, $y$ as a couple of solutions
$(y^+,y^-)$, cf. \eref{eq:def_y_pm_a}, \eref{eq:def_y_pm_b}, which are written
in the form \eref{eq:def_y_pm_c}. The coupling is performed by searching
for $x_l$, $x_r$ and $y$ such that \[ x_l^+(\Omega^+)=y(0), \quad
x_r^-(-\Omega^-)=y(\nu), \] or equivalently in terms of the perturbations 
  \begin{description}
    \item[($C$)] $v_l^+(a_l,u_l,\lambda)(\Omega^+) = w^+(b,u_l,\lambda)(0)$,\quad
      $v_r^-(a_r,u_r,\lambda)(-\Omega^-) = w^-(b,u_r,\lambda)(0)$.
  \end{description}
The orbit $x_m$ is the suspension of the orbit $y$ (see theorem~\ref{thm:lhc}). Note
that the transition time $\tau$ is essentially determined by the `length $\nu$'
of the orbit $y$.

For the actual coupling analysis inside the Poincar\'e section
$\Sigma_\gamma$, we have to impose a technical assumption.
\begin{hyp}
 \label{hyp:scaling}
{\rm All solutions of~\eref{eq:system} in a sufficiently small neighborhood of
 $q_l^+$ have the same transition time $\Omega^+$ from $\Sigma_l$ to
 $\Sigma_\gamma$.  Similarly, all solutions in a sufficiently small neighborhood of
 $q_r^-$ have the same transition time $\Omega^-$ from $\Sigma_\gamma$ to
 $\Sigma_r$.}
\end{hyp}
This can be achieved simultaneously by a scaling of the vector field in
a tubular neighborhood along $q_l^+$/$q_r^-$.  Note that this
scaling does not influence any of the previous results; see~\cite{Riess2008} for
a similar computation.

Hypothesis~\ref{hyp:scaling} 
guarantees that the points $q_l^+(u_l,\lambda)(\Omega^+) +
v_l^+(u_l,\lambda)(\Omega^+)$ and $q_r^-(u_r,\lambda)(-\Omega^-) +
v_r^-(u_r,\lambda)(-\Omega^-)$ are both in $\Sigma_\gamma$.
Further, this hypothesis allows to determine, cf. also \eref{eq:projectionplus} and \eref{eq:projectionminus}:
\begin{equation}\label{eq:projections}
\begin{array}{lcl}
\ker Q^+(u_l,\lambda)(0)&=&\Phi^+_l(u_l,\lambda)(\Omega^+,0)(W^-_l\oplus Z_l)
\\
&=&\ker P^+(u_l,\lambda)(\Omega^+),
\\[1ex]
\ker Q^-(u_r,\lambda)(0)&=&\Phi^-_r(u_r,\lambda)(\Omega^-,0)(W^-_r\oplus Z_r)
\\
&=&\ker P^-(u_r,\lambda)(\Omega^-),
\end{array}
\end{equation}
and~\eref{eq:def_imQ} implies
\begin{equation}\label{eq:imQ_subset}
 \begin{array}{ll}
\im Q^+(u_l,\lambda)(0) \subset \im P^+_{l}(u_l,\lambda)(\Omega^+),
\\[1ex]
\im Q^-(u_r,\lambda)(0) \subset \im P^-_{r}(u_r,\lambda)(-\Omega^-).
\end{array}
\end{equation}
An immediate consequence of~\eref{eq:projections} and~\eref{eq:imQ_subset} is
(cf. also the explanations following lemma~\ref{lem:exptrich}):
\begin{lem}\label{lem:iso_Ps}
 The restriction $P^+_{s,l}(u_l,\lambda)(\Omega^+)\left|_{{\rm im}\,
 Q^+(u_l,\lambda)(0)}\right.$ acts as an isomorphism ${\rm im}\,
 Q^+(u_l,\lambda)(0)\to {\rm im}\, P^+_{s,l}(u_l,\lambda)(\Omega^+)$. Moreover,
 for all $v^+\in Y_\gamma$ we have 
 \[
\begin{array}{l}
 P^+_{s,l}(u_l,\lambda)(\Omega^+)Q^+(u_l,\lambda)(0)v^+=P^+_{s,l}(u_l,\lambda)(\Omega^+)v^+,
\\[1ex]
(\id -Q^+(u_l,\lambda)(0))v^+=(\id-P^+_l(u_l,\lambda)(\Omega^+))v^+.
\end{array}
\] 
Similarly $P^-_{u,r}(u_r,\lambda)(-\Omega^-)\left|_{{\rm im}\,
Q^-(u_r,\lambda)(0)}\right.: {\rm im}\, Q^-(u_r,\lambda)(0)\to {\rm im}\,
P^-_{u,r}(u_,\lambda)(\Omega^-)$ is an isomorphism, and  for all $v^-\in
Y_\gamma$ we have \[
\begin{array}{l}
 P^-_{u,r}(u_r,\lambda)(-\Omega^-)Q^-(u_r,\lambda)(0)v^-=P^-_{u,r}(u_r,\lambda)(-\Omega^-)v^-,
\\[1ex]
(\id -Q^-(u_r,\lambda)(0))v^-=(\id-P^-_r(u_r,\lambda)(-\Omega^-))v^-.
\end{array}
\]
\end{lem}

Further, we use the notation
\[
{\mathcal U} := U_l \times U_r, \quad  {\mathcal U}\ni u= (u_l,u_r).
\]
The following lemma is a reformulation of theorem~\ref{thm:lhc} in terms of the
perturbances $v_l^\pm$, $v_r^\pm$ and $w^\pm$.
\begin{lem}
  \label{lem:couplingper}
Fix $\omega^+,\omega^->0$.
  There are constants $\Omega^-,\Omega^+>0$, $N\in\N$, $c>0$ such that
  for all $\abs{\lambda}<c$, $u\in {\mathcal U}$, $\norm{u}<c$,
  $\nu>N$, and for given sufficiently small $a_l^-,a_r^+\in \R^n$ there are
  $b\in Y_\gamma\times Y_\gamma$ and $a_l^+,a_r^-\in\R^n$
such that
  \begin{description}
    \item[($B_l^-$)] $(\id - P^-_l(u_l,\lambda)(-\omega^-)
      (v_l^-(a_l,u_l,\lambda)(-\omega^-) - a_l^-)=0$,
    \item[($B_r^+$)] $(\id - P^+_r(u_r,\lambda)(\omega^+) (v_r^+(a_r,u_r,\lambda)(\omega^+)
      - a_r^+)=0$,
    \item[($C$)] $v_l^+(a_l,u_l,\lambda)(\Omega^+) = w^+(b,u_l,\lambda)(0)$,\quad
      $v_r^-(a_r,u_r,\lambda)(-\Omega^-) = w^-(b,u_r,\lambda)(0)$.
  \end{description}
\end{lem}

\begin{proof}
We show that $\left(a_l^+,a_r^-,b^+,b^-\right)$ are uniquely determined in
\begin{equation*}
    \fl\Delta_{u,\lambda} := \im \left(\id - P_l^+(\Omega^+) \right) \times \im \left( \id - 
    P_r^-(-\Omega^-) \right)
    \times \im Q^+(0) \times \im Q^-(0).
  \end{equation*}
Note that all projections appearing in the definition of $\Delta_{u,\lambda}$ depend on
$(u_l,\lambda)$ or $(u_r,\lambda)$, respectively.  Throughout the proof we
suppose that $\left(a_l^+,a_r^-,b^+,b^-\right)$ belongs to $\Delta_{u,\lambda}$.

We construct $a_l^+$, $a_r^-$ and $b$ by solving an appropriate
fixed point equation.

As a consequence of hypothesis~\ref{hyp:scaling} we know that
$v_l^+(a_l,u_l,\lambda)(\Omega^+)\in Y_\gamma$. In what follows we suppress the
dependence on $a_l$, $u_l$, $\lambda$ and $b$ from our notation.  Therefore, in
accordance with lemma~\ref{lem:iso_Ps}, we find:
\[
\fl
 \begin{array}{l}
 v_l^+(\Omega^+)=Q^+(0)v_l^+(\Omega^+)+(\id-Q^+(0))v_l^+(\Omega^+)
\\[1ex]
\rule{3em}{0pt}=\left( P^+_{s,l}(\Omega^+)\left|_{{\rm im}\,Q^+(0)}\right. \right)^{-1}P^+_{s,l}(\Omega^+)v_l^+(\Omega^+)+
(\id-P^+_l(\Omega^+)v_l^+(\Omega^+)
\\[1ex]
\rule{3em}{0pt}=\left( P^+_{s,l}(\Omega^+)\left|_{{\rm im}\,Q^+(0)}\right. \right)^{-1}\, a^+_{\perp,l}\,+a^+_l.
\end{array}
\]
On the other hand, in accordance with lemma~\ref{lem:lambdalemma}:
\[
w^+(0)=Q^+(0)w^+(0)+(\id-Q^+(0))w^+(0)=b^++b^+_\perp.
\]
Hence, $v_l^+(a_l,u_l,\lambda)(\Omega^+) = w^+(b,u_l,\lambda)(0)$ if and only if 
\begin{equation}\label{eq:fix_part1}
 b^+=\left( P^+_{s,l}(\Omega^+)\left|_{{\rm im}\,Q^+(0)}\right. \right)^{-1}\, a^+_{\perp,l}=:\alpha^+_{\perp,l}
\quad
{\rm and}
\quad
a^+_l=b^+_\perp.
\end{equation}
In a similar way we find that $v_r^-(a_r,u_r,\lambda)(-\Omega^-) = w^-(b,u_r,\lambda)(0)$ if and only if
\begin{equation}\label{eq:fix_part2}
 b^-=\left( P^-_{u,r}(-\Omega^-)\left|_{{\rm im}\,Q^-(0)}\right. \right)^{-1}\, a^-_{\perp,r}=:\alpha^-_{\perp,r}
\quad
{\rm and}
\quad
a^-_r=b^-_\perp.
\end{equation}
Altogether, for fixed $u_l$, $u_r$ and $\lambda$ equations \eref{eq:fix_part1}
and \eref{eq:fix_part2} are equivalent to the fixed point equation
\begin{equation}\label{eq:fp_g}
\begin{array}{ll}
\left(a_l^+,a_r^-,b^+,b^-\right)&=
\left( b_\perp^+(b), \; b_\perp^-(b), \alpha_{\perp,l}^+(a_l), \; \alpha_{\perp,r}^-(a_r)\right)
\\[1ex]
&=:{\mathcal G}_{u,\lambda}\left((a_l^+,a_r^-,b^+,b^-),(a_l^-,a_r^+)\right),
\end{array}
\end{equation}
where we consider ${\mathcal G}$ as a mapping 
\begin{equation*}
 {\mathcal G_{u,\lambda}}:  \Delta_{u,\lambda} \times (\R^n\times\R^n)\to \Delta_{u,\lambda}.
\end{equation*}

To solve the fixed point equation \eref{eq:fp_g} we apply the Banach fixed point
theorem.  First we show that there is a
$\mathcal{G}_{u,\lambda}(\cdot,(a_l^-,a_r^+))$-invariant closed ball
$B(0,\varepsilon)\subset\Delta_{u,\lambda}$. Then we prove that
$\mathcal{G}_{u,\lambda}(\cdot,(a_l^-,a_r^+))$ is a contraction on
$B(0,\varepsilon)$.

Let $\tilde{c}_l(K)$ and $\tilde{c}_r(K)$ be the constants according to
\eref{eq:def_tilde_c}.  The subscripts $l$ and $r$ refer to $v_l$ and $v_r$,
respectively. With the constant $\tilde{\mathfrak{c}}$ related to $w$, cf.
lemma~\ref{lem:lambdalemma}, we
define 
\[ 
\epsilon=\epsilon(K)
:=\min\{\tilde{c}_l(K),\tilde{c}_r(K),\tilde{\mathfrak{c}}\}.  
\] 
Now we fix some
sufficiently large $K$. Then the estimates given in  \eref{eq:est_aperp} and
corollary~\ref{cor:lambdalemma_b} provide that $\Omega^+$, $\Omega^-$ and $N$ can be
chosen so large that for all $\nu> N$ and for all $\abs{a_l^-},
\abs{a_r^+}<\epsilon(K)$ the mapping ${\mathcal
G}_{u,\lambda}(\cdot,(a_l^-,a_r^+))$ leaves the closed ball
$B_{\Delta_{u,\lambda}}(0,\epsilon(K))$ invariant.  This remains true also for
all larger $\Omega^+$, $\Omega^-$ and $N$.

Due to \eref{eq:est_aperpder} and \eref{eq:est_bperpder} the mapping ${\mathcal
G}_{u,\lambda}(\cdot,(a_l^-,a_r^+))$ is also a contraction on
$B_{\Delta_{u,\lambda}}(0,\epsilon(K))$ (with increased $\Omega^\pm$
and $N$, if necessary).
\end{proof}

\begin{cor}\label{cor:couplingper}
 Let the assumptions of lemma~\ref{lem:couplingper} hold. Then
 $(a_l^+,a_r^-,b^+,b^-)$ depend smoothly on $(a_l^-,a_r^+,u,\lambda)$.
\end{cor}

\begin{proof}
 For fixed $(u,\lambda)$ the smooth dependence on $(a_l^-,a_r^+)$ follows by
 applying the implicit function theorem at a solution of \eref{eq:fp_g}.

To prove the smooth dependence on $(u,\lambda)$ we redo the proof
of lemma~\ref{lem:couplingper} to some extent. This time, however, we decompose
$v_l^+(\Omega^+)$ and $w^+(0)$ by means of $Q^+(0,0)(0)$ instead of
$Q^+(u_l,\lambda)(0)$. Similarly, we decompose $v_r^-(-\Omega^-)$ and $w^-(0)$ by
means of $Q^-(0,0)(0)$. In this way we get a fixed point equation in 
\[
\Delta:=\im(\id-Q^+(0))\times\im(\id-Q^-(0))\times\im Q^+(0)\times\im Q^-(0), 
\]
where all projections are considered at $(u_l,\lambda)=(0,0)$ or
$(u_r,\lambda)=(0,0)$, respectively. Note that $\Delta$ does not depend on
$(u,\lambda)$, and there is a $(u,\lambda)$-dependent isomorphism acting between
$\Delta_{u,\lambda}$ and $\Delta$.  This leads to a fixed point equation,
similar to \eref{eq:fp_g}, defined by a mapping \[ {\mathcal
G}:\Delta\times(\R^n\times\R^n)\times{\mathcal U}\times\R^m\to\Delta.  \]
Exploiting this fixed point equation yields the corollary.
\end{proof}

\begin{proof}[Proof of corollary~\ref{cor:lhc}]
  The statement of corollary~\ref{cor:lhc} follows immediately from
  lemma~\ref{lem:couplingper} with $a_l^- = a_r^+ = 0$, see also the proof of
  corollary~\ref{cor:linbvp} and remark~\ref{rem:infinity}.
\end{proof}

\subsection{Jump estimates}
\label{sec:jump_3}

Let the conditions of theorem~\ref{thm:lhc} hold, and let $(x_l,x_m,x_r)$ denote
the long Lin orbit segment. 
 According to \eref{eq:def_Xi} we define:
\begin{equation}\label{eq:def_Xi_lr}
  \fl\begin{array}{lcl}
  \Xi_l(\nu,a_l^-,a_r^+,u,\lambda) &:=& x_{l}(\nu,a_l^-,a_r^+,u,\lambda)(0) -
  x_{m}(\nu,a_l^-,a_r^+,u,\lambda)(0),\\
 \Xi_r(\nu,a_l^-,a_r^+,u,\lambda) &:=& x_{m}(\nu,a_l^-,a_r^+,u,\lambda)(\tau) -
  x_{r}(\nu,a_l^-,a_r^+,u,\lambda)(0).
\end{array}
\end{equation}
We now consider exemplarily the jump $\Xi_l$ within $\Sigma_l$ more closely.
For that purpose we write $\Xi_l$ in the form, cf. \eref{eq:jumpsplit} and \eref{eq:def_xi},
\begin{equation*}
 \Xi_l(\nu,a_l^-,a_r^+,u,\lambda)=\xi_l^\infty(u_l,\lambda) + \xi_l(\nu,a_l^-,a_r^+,u,\lambda),
\end{equation*}
where
\begin{eqnarray*}
  \xi_l^\infty(u_l,\lambda) &:=&  q_l^-(u_l,\lambda)(0) -
    q_l^+(u_l,\lambda)(0),\\
    \xi_l(\nu,a_l^-,a_r^+,u,\lambda) &:=& v_l^-(a_l,u,\lambda)(0) -
  v_l^+(a_l,u,\lambda)(0),
\end{eqnarray*}
with $a_l = (a_l^-,a_l^+(a_l^-,a_r^+,u,\lambda))$; cf.
lemma~\ref{lem:couplingper}.

Recall that we denote the leading stable Floquet multiplier of $\gamma$ by
$\mu^s$.

\begin{lem}
  \label{lem:jump}
Let the constants $a_l^-,a_r^+,u$ and $\lambda$ be in agreement with
theorem~\ref{thm:lhc}, and let $\nu$ be sufficiently large. Further, we assume
hypotheses~\ref{hyp:leading_ev}, \ref{hyp:dim_Z} and \ref{hyp:orth_decomp} hold
for the short Lin orbit segment defined by $q_l^{-/+}$ and $v_l^{-/+}$, and
we assume that the non-orbit-flip condition for $q_l^{+}$ holds, meaning that $q_l^{+}$ is not in the strong stable
manifold of $\gamma$.  Then 
  \begin{equation*}
    \eqalign{
\xi_l(\nu,a_l^-,a_r^+,u,\lambda)={\mathcal O}((\mu^u)^{-\nu}) 
+\mathcal{O}(|a_l^-|).
   }
  \end{equation*}
The $\mathcal{O}(\cdot)$-terms are valid for $\nu\to\infty$ or $|a_l^-|\to 0$, respectively.
\end{lem}

\begin{proof}
  Lemma~\ref{lem:jump_2} yields that $\xi_l(\nu,a_l^-,a_r^+,u,\lambda) =
  \mathcal{O}(|a_l^-|) + \mathcal{O}(|a_l^+|)$.  The coupling condition ($C$), see also \eref{eq:fp_g},
  yields that $a_l^+ = b_\perp^+$, and from
  lemma~\ref{lem:lambdalemma_b} we get $|b_\perp^+| = c^u(b,u,\lambda)
  (\mu^u)^{-\nu} + o(|\mu^u|^{-\nu})$.  
\end{proof}

In what follows we assume that $x_l$ and $x_r$ approach $\gamma_l$ and
$\gamma_r$, respectively.  For the jumps of the heteroclinic Lin orbit $x_l\cup
x_m\cup x_r$ connecting $\gamma_l$ and $\gamma_r$ (via $\gamma$)  we get:
\begin{cor}
  \label{cor:jump}
Let the heteroclinic Lin orbit $x_l\cup x_m\cup x_r$ be in agreement with
corollary~\ref{cor:lhc}, and let $\nu$ be sufficiently large. Further we assume
hypotheses~\ref{hyp:leading_ev}, \ref{hyp:dim_Z} and \ref{hyp:orth_decomp} hold
for the short Lin orbit segment defined by $q_l^{-/+}$ and $v_l^{-/+}$, and
we assume that the non-orbit-flip condition for $q_l^{+}$ holds, meaning that $q_l^{+}$ is not in the strong stable
manifold of $\gamma$.  Then there is a smooth function $c_l^u:
\mathcal{U}\times\R^m\to\R$ such that
  \begin{equation*}
    \eqalign{
   \langle z, \xi_l(\nu,u,\lambda)\rangle &= c_l^u(u,\lambda) (\mu^u)^{-\nu} 
    + o(|\mu^u|^{-\nu}),
    }
  \end{equation*}
where the $o(\cdot)$-term is valid for $\nu\to\infty$.
\end{cor}

\begin{proof}
 In accordance with corollary~\ref{cor:jump_2} and \eref{eq:jump_2} we have
\[
 \langle z, \xi_l(\nu,u,\lambda)\rangle=-\langle \Phi^+_l(0,\omega^+)^T(\id -P^+_l(u,\lambda)(0))^Tz,a_l^+\rangle +
o(\abs{a_l^+}).
\]
Using \eref{eq:fp_g} and \eref{eq:est_bperp} (in this order) yields
\[
\fl
\begin{array}{lcl}
 \langle z, \xi_l(\nu,u,\lambda)\rangle&=&-\langle \Phi^+_l(0,\omega^+)^T(\id -P^+_l(u,\lambda)(0))^Tz,b_\perp^+\rangle +
o(\abs{b_\perp^+})
\\[1ex]
&=& -c^u(b,u,\lambda)(\mu^u)^{-\nu}\langle \Phi^+_l(0,\omega^+)^T(\id
-P^+_l(u,\lambda)(0))^Tz,\hat b_\perp^+\rangle\\[1ex] 
&&+o((\mu^u)^{-\nu})
\end{array}
\] The notations $c^u$ and $\hat b_\perp^+$ are in accordance with the proof of
lemma~\ref{lem:lambdalemma_b}.  Note that both $b$ and $\hat b_\perp^+$ depend
on $(u,\lambda)$.  With that we finally get
\begin{equation}\label{eq:def_c_l^u}
 c_l^u(u,\lambda)=-c^u(b,u,\lambda)\langle \Phi^+_l(0,\omega^+)^T(\id -P^+_l(u,\lambda)(0))^Tz,\hat b_\perp^+\rangle.
\end{equation}
The smoothness of $c_l^u$ follows from the representation \eref{eq:def_c_l^u} ---
recall that all ingredients there depend smoothly on $(u,\lambda)$.
\end{proof}

\begin{cor}\label{cor:jump_a}
Let the assumptions of corollary~\ref{cor:jump} hold. Additionally let $n=3$.
Then the function $c_l^u$ from corollary~\ref{cor:jump} satisfies $c_l^u(0,0)\not=0$.
\end{cor}

\begin{proof}
 Consider the explicit representation \eref{eq:def_c_l^u} of $c_l^u$.

First we make clear that $c^u(0,0,0)\not=0$: For that we recall the arguments
justifying \eref{eq:def_c^u} and note that, in the present context ($n=3$ and
hence $\dim Y_\gamma=2$), hypotheses~\ref{hyp:approach_q} and
\ref{hyp:approach_b_perp} are automatically fulfilled.

Next we consider the second term on the right-hand side of \eref{eq:def_c_l^u}.
Due to \eref{eq:def_Z} and \eref{eq:projectionplus} we have $(\id
-P^+_l(0,0)(0))^Tz=z$.  Again due to $n=3$, and further due to the definitions
of $b_\perp^+$ and $Q^+$, cf. \eref{eq:def_b_perp}, and \eref{eq:def_imQ} it is
clear that $ \langle \Phi^+_l(0,\omega^+)^Tz,\hat b_\perp^+\rangle\not=0.  $
\end{proof}

\begin{rem}
  \label{rem:jump}
  The jump $\Xi_r$ in $\Sigma_r$ can be treated in a similar way.
\end{rem}

\section{Application to EtoP cycles}
\label{sec:applications}
In this section we apply the theory of the existence of long Lin orbits
developed in the previous sections to EtoP cycles.
We discuss bifurcations of $1$-homoclinic orbits to the
equilibrium in the neighborhood of the EtoP cycle, and we compare these results with
the numerical results of a concrete vector field from~\cite{Krauskopf2006, Krauskopf2008},
which also serves as the main motivating example for our studies.
Here we refer to homoclinic orbits to $E$ (near the cycle under consideration) making
only one excursion to $P$ as $1$-homoclinic orbits.

Let the EtoP cycle consist of a hyperbolic equilibrium $E$, a hyperbolic periodic orbit $P$ and heteroclinic orbits $q_l$, connecting $E$ to $P$, and $q_r$
connecting $P$ to $E$. Then, in the language of the
previous sections, such a EtoP cycle can be considered as an orbit segment
$E\cup q_l\cup P\cup q_r\cup E$.  The 1-homoclinic orbits to $E$ then can be
found among the homoclinic Lin orbits near this orbit segment.  Therefore, the
bifurcation equations for detecting 1-homoclinic orbits to $E$ are generated by
making the jump functions $\Xi_l$ and $\Xi_r$ equal to zero, cf.
\eref{eq:def_Xi_lr}:
\begin{equation}\label{eq:def_bif_eqn}
 \Xi_l(\nu,u,\lambda)=0,\quad \Xi_r(\nu,u,\lambda)=0.
\end{equation}
Note that in the present context $a_l^-$ and $a_r^+$ are zero, the corresponding
$\xi_l$ and $\xi_r$ are given by corollary~\ref{cor:jump}.

In our analysis we distinguish two types of EtoP cycles. First, we consider
codimension-one cycles characterized by a robust heteroclinic connection $q_l$
and a connection $q_r$ that splits up with positive speed while moving the family
parameter $\lambda$. We prove an accumulation of 1-homoclinic orbits to $E$ near
the original EtoP cycle in the following sense: for each sufficiently large
$\nu\in\N$ there is a $\lambda_\nu$ for which a $1$-homoclinic orbit exists. This
homoclinic orbit performs $\nu$ rotations along $P$ before returning to $E$.
The $\lambda_\nu$ accumulate at $\lambda=0$, the critical parameter value for
which the cycle exits.

Second, we study codimension-two cycles, where in comparison with the previous
one, we `merely modify' the behavior near $q_l$: We demand that the unstable
manifold of $E$ and the stable manifold of $P$ have a quadratic tangency along
$q_l$. The parameters $\lambda_l$ and $\lambda_r$ unfold the orbits $q_l$ and
$q_r$ independently. Generically, 1-homoclinic orbits are still codimension-one
objects --- hence they are expected to appear along curves in the parameter
space. Indeed, for each $\nu$ we find those orbits on a curve $\kappa_\nu$ in
$(\lambda_l,\lambda_r)$-space. Each curve has a turning point which tends
to the critical parameter value $(\lambda_l,\lambda_r)=(0,0)$.  Here again $\nu$
counts the rotations of the 1-homoclinic orbits near $P$.

Our analysis is local in nature. However, the local phenomena described above
are part of a global scenario observed numerically in several examples. In
parameter space $1$-homoclinic orbits can be continued along a curve which snakes
between two curves (which are related to a `quadratic tangency at $q_l$') and
accumulates on a curve segment for which the EtoP cycle exists, an example will
be introduced in the next section.

Finally we mention that in the recent paper \cite{Champneys} similar
phenomena have been discussed.

\subsection{Unfolding of a saddle-node Hopf bifurcation with global reinjection
mechanism} \label{sec:example}

\begin{figure}[b] 
  \begin{center} 
  \includegraphics[scale=.8]{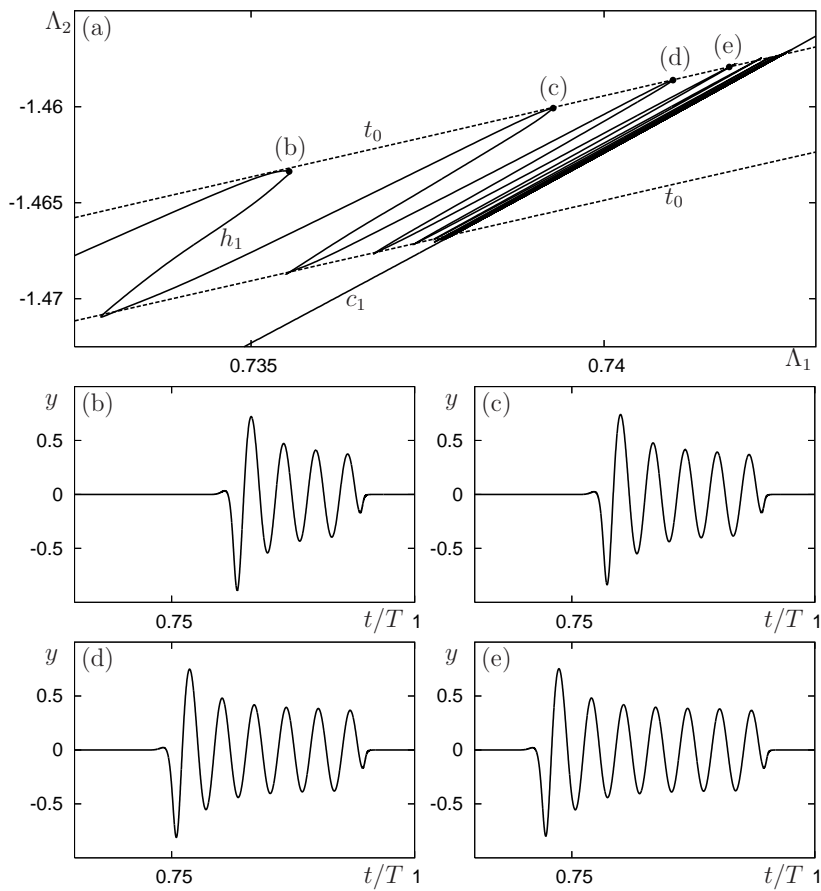} 
  \caption{\label{fig:bif_h1b}Panel (a) shows a detail of the bifurcation
  diagram of~\eref{eq:injsystem} in the $(\Lambda_1,\Lambda_2)$-plane.  Shown
  are the curve
  $h_1$ of a codimension-one homoclinic orbit to $E$, the curve $c_1$ of a codimension-one
  EtoP connection from $E$ to $P$, and the curve $t_0$ of tangencies of a
  codimension-zero EtoP connection from $P$ back to $E$.  Panels (b)--(e)
  show the relevant part of a time-versus-$y$ plot of selected homoclinic orbits
  on $h_1$ that illustrate how they take more rotations close to $P$ as
  they approach the complete EtoP cycle at the intersection of $c_1$ and
  $t_0$.  Here, $T$ is the total integration time of the computed orbit
  segments.} 
  \end{center} 
\end{figure}

In this section we consider a three-dimensional model
vector field that was introduced in~\cite{Krauskopf2006}, the numerical results
concerning the EtoP cycle presented here are from~\cite{Krauskopf2008}.  The vector field
describes the dynamics near a saddle-node Hopf bifurcation in the presence of a
global reinjection mechanism.  This type of dynamics with reinjection can be
found, for example, in laser systems \cite{tollenaar,physrep,zimmermann}, in
dynamo theory \cite{ashwin} and, more generally, near weak
resonances~\cite[chapter 4.3.2]{vitolo}.  The vector-field model can be written
in the form
\begin{eqnarray}
  \label{eq:injsystem} 
\fl
  \begin{array}{rcl}
    \dot{x} &=& \Lambda_1 x -
  \omega y - (\alpha x - \beta y) \sin \varphi - (x^2 + y^2)x
+ d ( 2\cos\varphi + \Lambda_2)^2,\quad\quad 
\\[1ex]
  \dot{y} &=& \Lambda_1 y + \omega x - (\alpha y + \beta x) \sin \varphi - (x^2 +
  y^2)y 
+ f ( 2\cos\varphi +
  \Lambda_2)^2,\quad\quad\ 
\\[1ex] 
\dot{\varphi} &=& \Lambda_2 +
  s(x^2+y^2)+2\cos\varphi+c(x^2+y^2)^2, 
\end{array} 
\end{eqnarray}
where $\Lambda_1$ and $\Lambda_2$ are the unfolding parameters of the saddle-node Hopf
bifurcation. The parameters $\omega$, $\alpha$, $\beta$, $s$, $c$, $d$ and $f$
determine the type of unfolding and we keep them fixed throughout at
\begin{eqnarray*} 
\omega = 1.0,\;\alpha = -1.0,\;\beta = 0,\;s = -1.0,\;c =
  0,\;d = 0.01,\;f = \pi d.  
\end{eqnarray*}
This choice corresponds to the unfolding of type $A$ that was studied in
\cite{Krauskopf2006}, where more details can be found.  The variable $\varphi$
is $2\pi$-periodic and global reinjection is realized by trajectories that
connect a neighborhood of a saddle-node Hopf point with one of its symmetric
copies. Hence,
a global reinjection corresponds to a large excursion near the circle
$\mathbb{S}^1 = \{ x=y=0 \}$.  Note that this circle is not invariant because
$d\not=0$ and $f\not=0$ (where rational ratios are avoided). 

As already shown in~\cite{Krauskopf2006}, the system has a wide variety of
homoclinic orbits of saddle-focus equilibria involving one or more global
reinjections.  The most interesting one in the present context is the homoclinic
orbit $h_1$ to the saddle-focus equilibrium $E=(0,0,\arccos(\nu_2/2))$ with one
global excursion, which accumulates on a curve segment in parameter space, while
the orbit itself accumulates on an EtoP cycle connecting $E$ and a periodic
orbit $P$.  \Fref{fig:bif_h1b} (a) shows the relevant part of the bifurcation
diagram of \eref{eq:injsystem} where the accumulation of $h_1$ takes place.
Both curves $t_0$ are continuation curves of the codimension-one heteroclinic
orbit $q_l$ connecting $E$ to a periodic orbit $P$. Here the codimension is
characterized by a quadratic tangency of the unstable manifold of $E$ and the
stable manifold of $P$.  The curve $c_1$ is the continuation curve of the
heteroclinic orbit $q_r$ connecting $P$ to $E$.  The dimensions of the unstable
manifold of $P$ and of the stable manifold of $E$ add up to the space dimension.
Hence, $q_r$ is also a codimension-one heteroclinic orbit.

The complete heteroclinic EtoP cycle is given by $E \cup q_l \cup P \cup q_r$.
Panels (b)--(e) show time-versus-$y$ plots of selected homoclinic orbits along
$h_1$ that illustrate how the homoclinic orbit accumulates on the EtoP cycle as
the bifurcation curve $h_1$ accumulates on the segment of $c_1$ where the
complete EtoP cycle exists.

Our goal here is to explain the accumulation process of $h_1$ analytically.  More
precisely, we are going to show two features of $h_1$.  First, we consider a
one-parameter family along a curve $(\Lambda_1(\lambda),\Lambda_2(\lambda))$
somewhere in the middle between the two curves $t_0$, and show that the
$1$-homoclinic orbit $h_1$ to $E$ accumulates at discrete points
on that parameter line.  Second, using an unfolding of the tangencies $t_0$ at
the intersection point of $t_0$ with $c_1$, we explain the shape of $h_1$ near
the turning points (near the points labeled (b)--(e) in
\fref{fig:bif_h1b} (a)) during the snaking process.

\subsection{Accumulation of homoclinic orbits near EtoP cycles}
Consider a one-parameter family of ODE \eref{eq:system}, that is $m=1$.
We assume that for $\lambda=0$ there is a heteroclinic EtoP cycle consisting of
a  hyperbolic equilibrium $E$,
a hyperbolic periodic orbit $P$ and heteroclinic orbits $q_r$ and $q_l$
connecting $P$ to $E$ and 
$E$ to $P$, respectively.
In accordance with the
notation in section~\sref{sec:twolinks}, we have $E=\gamma_l \equiv \gamma_r$
and $P = \gamma$.  

The aim of this section is to study homoclinic bifurcations from the given
heteroclinic EtoP cycle under some additional genericity conditions.

Throughout we consider the system for $\lambda\in(-c,c)$, $c$ sufficiently small.
We assume:
\begin{description}
 \item[(C1)] $\dim W^s(E) = k$ and $\dim W^u(P) = n-k$;
 \item[(C2)] $W^s(E)$ and $W^u(P)$ intersect in an isolated connecting orbit $q_r$;
\item[(C3)] The extended manifolds 
$\bigcup_{\lambda\in(-c,c)} W_\lambda^u(P)\times \{ \lambda \}$ 
and
$\bigcup_{\lambda\in(-c,c)} W_\lambda^s(E)\times \{ \lambda \}$ 
intersect transversally in $\R^{n}\times\R$;
\item[(C4)] $W^s(P)$ and $W^u(E)$ intersect transversally along $q_l$; 
\item[(C5)] The leading stable Floquet multiplier $\mu^s(\lambda)$ of $P$
    is real and simple;  
\item[(C6)] $q_l$ and $q_r$ approach $E$ and
    $P$ along the leading stable/unstable directions (non-orbit-flip condition).
\end{description}

\begin{rem}\label{rem:rel_to_appl}
 The following one-parameter subfamily of \eref{eq:injsystem} satisfies the
 conditions (C1) -- (C6): Let $\kappa = (\Lambda_1(\lambda),\Lambda_2(\lambda))$ and
 $(\Lambda_1(0),\Lambda_2(0))\in \hat c_1$. Here $\hat c_1$ denotes the part of
 $c_1$
 between the intersections of $c_1$ and the curves $t_0$. Further we assume that
 $\kappa$ and $c_1$ intersect transversally. 
\end{rem}

In order to apply the theory that we developed in the previous sections, we
introduce cross-sections $\Sigma_{l/r}$ of $q_{l/r}$.  Conditions (C1)--(C4)
ensure that $\dim U_l = \dim U_r=0$.  Therefore, corollary~\ref{cor:lhc} tells
us that for each sufficiently small $\lambda$ there is a unique homoclinic Lin
orbit $x = (x_l,x_m,x_r)$ connecting $E$ to itself.  The actual 1-homoclinic
orbits relate to solutions of the bifurcation equation \eref{eq:def_bif_eqn}.
However, because of (C1)--(C4) we have $\dim Z_l = 0$ and $\dim Z_r = 1$.  This
means that the (unique) homoclinic Lin orbit has exactly one discontinuity, and
this discontinuity is located inside $\Sigma_r$. In other words, $\Xi_l$ is identically zero
and the bifurcation equation for 1-homoclinic orbit reduces to (note that
no $u$ is involved) 
\[ 
\Xi_r(\nu,\lambda)=0.  
\] 
The jump function $\Xi_r$ is
defined by \eref{eq:def_Xi} and \eref{eq:jumpsplit} 
\[ 
\Xi_r(\nu,\lambda) :=
x_m(\nu,\lambda)(\tau)-x_r(\nu,\lambda)(0) = \xi_r^\infty(\lambda) +
\xi_r(\nu,\lambda).  
\] 
Condition (C3) yields that the manifolds $W^u(P)$ and
$W^s(E)$ split with non-vanishing velocity, in other words,
$D\xi_r^\infty(0)\not=0$.  Hence there is a parameter transformation such that
\begin{equation}\label{eq:lambda_rep}
 \xi_r^\infty(\lambda) = \lambda.
\end{equation}
For the remaining term $\xi_r(\nu,\lambda)$ we may employ corollary~\ref{cor:jump}:
\[
\xi_r(\nu,\lambda) = c_r^s(\lambda)(\mu^s)^\nu + o(|\mu^s|^\nu).
\]
Combining these terms yields the following lemma.
\begin{lem} 
  \label{lem:acchom} Assume (C1)--(C6) and \eref{eq:lambda_rep}.
  Then the jump function $\Xi_r(\nu,\lambda)$ can be written as
  \begin{equation}\label{eq:rep_Xi_r_app1}
    \Xi_r(\nu,\lambda) = \lambda + c_r^s(\lambda) \left(
    \mu^s(\lambda) \right)^\nu + o\left( | \mu^s(\lambda) |^\nu
    \right),
  \end{equation}
  where $c_r^s(\lambda): \R \to \R$ is smooth, and the
  $o(\cdot)$-term is valid for $\nu\to\infty$.
\end{lem}

A direct consequence of this lemma is the following corollary:
\begin{cor}\label{cor:Xi_r_zeros}
Under the assumptions of lemma~\ref{lem:acchom}
  there is a constant $N\in\N$ such that for all $\nu\in\N$, $\nu>N$, there is
  a $\lambda_\nu$ such that $\Xi_r(\nu,\lambda_\nu) = 0$.  Moreover,
  $\lambda_\nu$ tends to $0$ as $\nu\to\infty$.
\end{cor}
Note that the zeros of $\Xi_r$ correspond to $1$-homoclinic orbits to $E$.
Hence,
corollary~\ref{cor:Xi_r_zeros} says that for each sufficiently large natural
$\nu$ there is a 1-homoclinic orbit with $\nu$ rotations near $P$.  Further, we
see that in parameter space these orbit accumulate at $\lambda=0$ --- in state
space they accumulate onto the original EtoP cycle.

For more precise assertions we need to know that $c_r^s(0)\not=0$, which is true
if $n=3$, see in the proof of corollary~\ref{cor:jump_a}. Therefore we get:
\begin{cor}\label{cor:Xi_r_zeros_n=3}
 Let $n=3$. Assume further (C1)--(C6) and \eref{eq:lambda_rep}.
\begin{enumerate}
 \item If $\mu^s>0$, then there is a monotonically increasing/decreasing (if
   $c_r^s(0)>0$/$c_r^s(0)<0$) sequence $(\lambda_\nu)$, $\lambda_\nu\to 0$, such
   that $\Xi_r(\nu,\lambda_\nu) = 0$.  
 \item If $\mu^s<0$, then there is an
   alternating sequence $(\lambda_\nu)$, $\lambda_\nu\to 0$, such that
   $\Xi_r(\nu,\lambda_\nu) = 0$.
\end{enumerate}
The order in which $\lambda_\nu$ approaches $0$ for $\nu\to\infty$ is given by $\mu^\nu$.
\end{cor}

A subfamily of \eref{eq:injsystem} as described in remark~\ref{rem:rel_to_appl}
is related to (i) of corollary~\ref{cor:Xi_r_zeros_n=3}.  Altogether, in respect
to \eref{eq:injsystem} corollary~\ref{cor:Xi_r_zeros_n=3} explains the
accumulation process of $h_1$ along a line $\kappa$ according to
remark~\ref{rem:rel_to_appl}.  But it neither explains the global snaking
behavior of $h_1$ nor the local behavior of $h_1$ near the turning points, cf.
(b) -- (e) in panel (a) of figure~\ref{fig:bif_h1b}, which we
consider in the following section.

\subsection{Homoclinic orbits near degenerate EtoP cycles}

In this section we consider a codimension-two EtoP cycle. For that we modify in
the formerly introduced EtoP cycle only the heteroclinic orbit $q_l$. Here we
assume that along $q_l$ the unstable manifold of $E$ and the stable manifold of
$P$ do no longer intersect transversally but have a quadratic tangency. We use
parameters $\lambda_l$ and $\lambda_r$ to unfold the codimension-one
heteroclinic orbits $q_l$ and $q_r$, respectively, and write $\lambda
=(\lambda_r,\lambda_l)\in\R^2$.  Here $\lambda_r$ plays the same role as
$\lambda$ in the previous section, and $\lambda_l$ moves the manifolds $W^u(E)$
and $W^s(P)$ against each other in a direction which is orthogonal to the sum of
the tanget spaces of these manifolds.

We consider system~\eref{eq:system} for $\abs{\lambda}\in(-c,c)$, $c$
sufficiently small.  In detail we assume the following: We adopt the assumptions (C1), (C2)
and (C6) from the previous section as (C1'), (C2') and (C6'), respectively.  In
(C3) we only replace $\lambda$ by $\lambda_r$:
\begin{description}
  \item[(C3')] the extended manifolds 
$\cup_{\lambda_r\in(-c,c)} W_{(\lambda_r,0)}^u(P) \times \{ \lambda_r\}$
and
$\cup_{\lambda_r\in(-c,c)}W_{(\lambda_r,0)}^s(E) \times \{ \lambda_r \}$ 
intersect transversally in $\R^n\times\R$.
\end{description}
\begin{description}
  \item[(C4')] $W^s(P)$ and $W^u(E)$ have (along $q_l$) a quadratic tangency, and 
\\
the extended manifolds 
$\cup_{\lambda_l\in(-c,c)} W_{(0,\lambda_l)}^u(E) \times \{ \lambda_l\}$ 
and
$\cup_{\lambda_l\in(-c,c)}W_{(0,\lambda_l)}^s(P) \times \{ \lambda_l \}$ 
intersect transversally in $\R^n\times\R$.
\end{description}
In contrast to (C5), here we also have an assumption on the leading unstable Floquet multiplier
\begin{description}
  \item[(C5')] the leading stable and unstable Floquet multipliers
    $\mu^{s/u}(\lambda)$ of $P$ are real and simple.
\end{description}

\begin{rem}
Consider \eref{eq:injsystem}.
There is a parameter transformation 
$(\lambda_r,\lambda_l)\leftrightarrow (\Lambda_1,\Lambda_2)$, with $(\Lambda_1(0),\Lambda_2(0))
  \in c_1 \cap t_0$, such that in the new parameters~\eref{eq:injsystem} satisfies (C1') -- (C6').
\end{rem}

The above assumptions imply the following dimensions of the involved linear
subspaces: $\dim Z_r = 1$, $\dim U_r = 0$, $\dim Z_l = 1$, $\dim U_l = 1$.
Hence, the variable $u = u_l$ appears in the jump function and in the
bifurcation equation, respectively.

The jump function $\Xi$ now consists of two parts, each representing one jump:
\begin{eqnarray*}
  \Xi(\nu,u,\lambda) = \left( 
\begin{array}{c} 
\Xi_r(\nu,u,\lambda)
\\
\Xi_l(\nu,u,\lambda) 
\end{array}
                      \right),
\end{eqnarray*}
where
$\Xi_{r/l}(\nu,u,\lambda)=\xi^\infty_{r/l}(u,\lambda)+\xi_{r/l}(\nu,u,\lambda)$.
Indeed $\xi^\infty_{r}$ depends only on $\lambda_r$ --- more precisely it has
(after an appropriate parameter transformation) the form, see also~\eref{eq:lambda_rep},
\begin{equation}\label{eq:lambda_rep1}
 \xi_r^\infty(\lambda_r) = \lambda_r.
\end{equation}
Further $\xi^\infty_{l}$ depends only on $\lambda_l$ and $u$, and
the quadratic tangency within $\Sigma_l$ can be modeled by 
\begin{equation}\label{eq:lambda_rep2}
 \xi_l^\infty(u,\lambda_l) = \lambda_l-u^2.
\end{equation}
Altogether this yields
\begin{lem}\label{lem:turning} 
Assume (C1')--(C6') and \eref{eq:lambda_rep1}, \eref{eq:lambda_rep2}.
    Then the bifurcation equation for 1-homoclinic orbits can be written as 
  \begin{equation} \label{eq:bif_eqn_1-hom}
\fl
    \Xi(\nu,u,\lambda) = \left(
    \begin{array}{c}\lambda_r\\
      \lambda_l-u^2
    \end{array}\right) + \left(
      \begin{array}{c} c^s_r(u,\lambda) \left( \mu^s(\lambda) \right)^\nu +
        o\left( |\mu^s(\lambda)|^{\nu} \right) \\
        c^u_l(u,\lambda) \left(
        \mu^u(\lambda) \right)^{-\nu} + o\left( | \mu^u(\lambda)
        |^{-\nu} \right)
      \end{array}\right)=0,
  \end{equation}
  where $c^{s/u}:\R\times\R^2 \to \R$, are smooth and the
  $o(\cdot)$-terms are valid for $\nu\to\infty$.
\end{lem}

\begin{figure}[ht]
  \begin{center} \includegraphics[scale=.8]{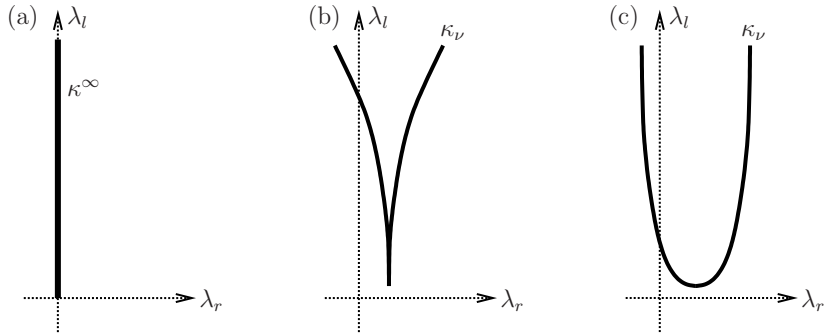} 
  \caption{\label{fig:kappa}Solution curves of \eref{eq:bif_eqn_cycles} and
  \eref{eq:bif_eqn_1-hom} in the $(\lambda_r,\lambda_l)$-plane for fixed $\nu$.
  Panel (a) shows the solution curve $\kappa^\infty$ of the unperturbed equation
  \eref{eq:bif_eqn_cycles}, which is a simple parabola in the
  $(\lambda_l,u)$-plane and a ray in the $(\lambda_r,\lambda_l)$-plane that is
  covered twice as $u$ is varied.  Panels (b) and (c) show possible
  perturbations of $\kappa^\infty$.} 
  \end{center}
\end{figure}

Solutions of the `unperturbed equation'
\begin{equation}\label{eq:bif_eqn_cycles}
  \left(
    \begin{array}{c}
      \lambda_r \\ \lambda_l-u^2
    \end{array}
  \right)
=0
\end{equation}
correspond to EtoP cycles near the original one. Solutions
$(\lambda_r,\lambda_l)(u)$ of \eref{eq:bif_eqn_cycles} are displayed in
figure~\ref{fig:kappa}, panel (a). For each $\lambda_l>0$ there are two
different heteroclinic orbits connecting $E$ to $P$. Therefore, for each nonzero
$\lambda\in\kappa^\infty$ there are two different EtoP cycles satisfying (C1) --
(C6) (along a curve intersecting $\kappa^\infty$ transversally).  So, according
to corollary~\ref{cor:Xi_r_zeros}, we expect for fixed $\nu$ and $\lambda_l>0$
two different $1$-homoclinic orbits $H_1(\nu,\lambda_l)$ and $H_2(\nu,\lambda_l)$
to $E$ with $\nu$ rotations near $P$. Both orbits can be continued in parameter
space.  Amazingly, they are located on the same continuation curve.  Indeed, solutions
of \eref{eq:bif_eqn_1-hom} for fixed $\nu$ are small perturbations of the
solutions of \eref{eq:bif_eqn_cycles}. In parameter space we find those
solutions on perturbations $\kappa_\nu$ of $\kappa^\infty$ as displayed in
figure~\ref{fig:kappa} (b) or (c).  Both, $H_1(\nu,\lambda_l)$ and
$H_2(\nu,\lambda_l)$ are on the same $\kappa_\nu$, but on different branches.
For decreasing $\lambda_l$, the orbits $H_1(\nu,\lambda_l)$ and $H_2(\nu,\lambda_l)$
finally `merge' in the vertex $\lambda_\nu$ of $\kappa_\nu$.

Generically one expects a perturbation of $\kappa^\infty$ as depicted in
figure~\ref{fig:kappa} (c). Below we show that in three-dimensional state
space such a parabola like curve will indeed appear --- as the numerical
computations suggest.

\begin{cor}\label{cor:app_turning} 
Let $n=3$. Assume further (C1')--(C6') and \eref{eq:lambda_rep1},
\eref{eq:lambda_rep2}.  For each (sufficiently large) $\nu\in\N$ there is a
solution curve of $\kappa_\nu=(\lambda_r,\lambda_l)(u)$ of the bifurcation
equation~\eref{eq:bif_eqn_1-hom} for 1-homoclinic orbits.  Further, there is a unique $u_\nu$ (for
each $\nu$) such that $D\lambda_l(u_\nu)=0$ and
$D\lambda_r(u_\nu)\not=0$.  The points
$\lambda_\nu:=(\lambda_r,\lambda_l)(u_\nu)$ accumulate at $\lambda=0$, and the
curvature of $\kappa_\nu$ in $\lambda_\nu$ tends to infinity as $\nu$ tends to
infinity.
\end{cor}

This corollary gives an explanation of the shape of $h_1$ (in
figure~\ref{fig:bif_h1b} (a))  locally around the turning points.  It also
provides information about the exponential rates with which the turning points
accumulate to $c_1\cap t_0$, namely they are given by the stable Floquet multiplier in
the $\lambda_r$-direction, and by the unstable Floquet multiplier in the
$\lambda_l$-direction.
Similar numerical studies in~\cite{Champneys} reveal that this mechanism also occurs in
other systems.

A possible arrangement of curves $\kappa_\nu$ is displayed in
figure~\ref{fig:parabolas}. This picture verifies the shape of the curve $h_1$
in figure~\ref{fig:bif_h1b} near the  points (b) -- (e) analytically.
The dashed line in figure~\ref{fig:parabolas} (the codimension-one line of the
heteroclinic orbit $q_l$) corresponds to the upper curve $t_0$ in
figure~\ref{fig:bif_h1b}. Note that the vertices of the curves
$\kappa_\nu$ are actually not located on this line. The fact that in
figure~\ref{fig:bif_h1b} the points (b) -- (e) are seemingly on $t_0$
is due to the large absolute value of the unstable Floquet multiplier and the
resulting quick convergence to $t_0$.

\begin{figure}[ht]
  \begin{center} \includegraphics[scale=.8]{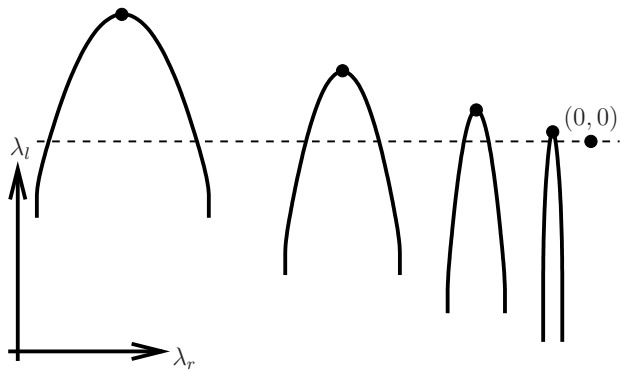}
  \caption{\label{fig:parabolas}The parabolas $\Xi(\nu,u,\lambda) = 0$ for
  increasing values of $\nu$ in the $(\lambda_r,\lambda_l)$-plane.  The vertices of
  the parabolas approach $(0,0)$; the order of the displacement in the $\lambda_r$-direction is
  given by $(\mu^s)^{\nu}$ and the order of the displacement in
  the $\lambda_l$-direction is given by $(\mu^u)^{-\nu}$.}
  \end{center}
\end{figure} 

\begin{proof}
 Consider the bifurcation equation \eref{eq:bif_eqn_1-hom}. First we note that
 the derivatives of $\xi_{l/r}$ with respect to $u$ and $\lambda$ admit the same
 estimate as given in corollary~\ref{cor:jump} for $\xi_l$; we refer to
 \cite{Knobloch2004} or \cite{Sandstede1993} for similar assertions including
 proofs.  So, using contraction arguments, we can solve $\Xi(\nu,u,\lambda)=0$
 for $\kappa_\nu:=\lambda(u,\nu)$ for sufficiently large $\nu$.  With
\begin{equation}\label{eq:der_c^s_r}
 D_1c^s_r(0,0)\not=0
\end{equation}
follows the existence of vertices of $\kappa_\nu$. Again using
\eref{eq:der_c^s_r} we can write $\kappa_\nu$ as
$\lambda_l=\lambda_l(\lambda_r)$.  From that representation one easily reads off
the assertion concerning the curvature.

It remains to verify \eref{eq:der_c^s_r}. Analogously to \eref{eq:def_c_l^u}, we find
\begin{equation}\label{eq:def_c_r^s} 
 c_r^s(u,\lambda)=c^s(b,u,\lambda)\langle\Phi_r^-(0,-\omega^-)^T(\id-P_r(\lambda)(0))^Tz,\hat b_\perp^-\rangle.
\end{equation}
Note that in the present context ($n=3$) the scalar product on the right-hand
side of \eref{eq:def_c_r^s} is different from zero, and the quantities within
the scalar product do not depend on $u$ ($\dim U_r=0$). To verify that for $\hat
b_\perp^-$ recall the definition \eref{eq:def_b_perp} of $b_\perp^-$. Note that
$Q^-$ here does not depend on $u$ ($\dim U_r=0$) and $\dim \im
(\id-q^-(\lambda))=1$ ($n=3$). Since $\hat b_\perp^-=b_\perp^-/\abs{b_\perp^-}$
we get that $\hat b_\perp^-$ does depend neither on $u$ nor on $b$.

So $D_1c_r^s(0,0)$ is different from zero if and only if $D_uc^s(0,0,0)\not=0$.
Similarly to \eref{eq:def_c^u}, we have $
c^s(b,u,\lambda)=\langle\eta^-(b,u,\lambda),\eta^s(u,\lambda)\rangle.  $
Actually, here $\eta^-$ does not depend on $(b,u)$. Further, note that $\eta^-$
is related to the asymptotics of $\Psi^-(0,-\nu^-)(\id-Q^-(0))^T\hat b_\perp^-$
(see in the proof of lemma~\ref{lem:lambdalemma_b}) and none of these terms
depends on $u$.  So $c^s=c^s(u,\lambda)$, and 
\begin{equation}\label{eq:der_c^s}
 D_1c^s(0,0)=\langle\eta^-(0),D_1\eta^s(0,0)\rangle.
\end{equation}
Roughly speaking, $\eta^s(u,\lambda)$ is related to the asymptotics of
$q_d^+(u,\lambda)(\cdot)$; see again the proof of
lemma~\ref{lem:lambdalemma_b}.

First we make clear that $D_1\eta^s(0,0)$ is different from zero. For this we
assume that the traces of $W^s(P)$ in both $\Sigma_l$ and $\Sigma_P$ are flat
(this can always be achieved by appropriate transformations). Then
$W^s(P)\cap\Sigma_l$ coincides with $U$, and
$q^+_d(u,0)(0)=\phi^{\Omega^+}(0)(u)$; compare hypothesis~\ref{hyp:scaling}.
Actually, this mapping can be considered as a mapping $\R\to\R$. Since
$\phi^{\Omega^+}(0)(\cdot)$ is a diffeomorphism ($\R^2\to\R^2$) also the above
considered restriction to $U$ is a diffeomorphism, and 
\[
D_1q^+_d(0,0)(0)\not=0.  
\] 
In the present context we have, see
\cite{Knobloch2004}, 
\[
\eta^s(u,\lambda)=\lim_{n\to\infty}(D_1\Pi(0,\lambda))^{-n}q^+_d(u,\lambda)(n).
\] 
From that representation we conclude that with $D_1q^+_d(0,0)(0)\not=0$ also
$D_1\eta^s(0,0)\not=0$.

Finally note that $\eta^s(u,0)\in T_{\tilde{p}}W_\Pi^s(\tilde{p})=W^s(P)\cap\Sigma_P$
(see above, $\tilde{p}$ denotes the $\Pi$-equilibrium $P\cap\Sigma_P$), which is
one-dimensional. Hence $D_1\eta^s(0,0)$ points in the same direction as
$\eta^s(0,0)$ and, because of \eref{eq:der_c^s}, we have $D_1c^s(0,0)\not=0$ (see
also the justification of \eref{eq:def_c^u}). Therefore \eref{eq:der_c^s_r}
holds.
\end{proof}

\section{Conclusions and outlook}

We adapted Lin's method to heteroclinic chains involving periodic orbits. The
main emphasis was on the coupling of two short Lin orbits near a periodic
orbit. In this way, we also achieved estimates of the jump functions (Lin
gaps), which are essential for detecting actual orbits near the primary chain
among the Lin orbits.

We employed our results to study 1-homoclinic orbits to the equilibrium near a
given EtoP cycle.  In particular we gave an analytical justification of some
local phenomena in the course of the (global) snaking behavior of the
continuation curve of $1$-homoclinic orbits.

A complete analytical description of the snaking behavior is still an open
problem; a global assumption on the behavior of the stable and unstable
manifolds of $E$ and $P$, similar to that used in~\cite{BKLSW2008}, is necessary
for such an analysis. In~\cite{Champneys} such an assumption has been used for a
geometric explanation of the mentioned global snaking phenomenon. However, also
these considerations are bound to $\R^3$.

Another interesting point is the more complete description of the dynamics near
an EtoP cycle, such as existence of $N$-homoclinic (or $N$-heteroclinic) orbits
to $E$ or $P$, periodic orbits, or shift dynamics.

\ack

The authors thank Bernd Krauskopf for helpful discussions and comments.  TR
acknowledges the financial support and hospitality of the Bristol Centre for Applied
Nonlinear Mathematics at the Department of Engineering Mathematics, University
of Bristol,
during several long research visits and during a 6-month research position in
2008.



\end{document}